\documentclass[12pt,reqno]{amsart}
\usepackage[margin=2.5cm]{geometry}
\usepackage{amsmath,amssymb,amsfonts,amsthm,pifont,mathtools,mathrsfs}
\usepackage[english]{babel}
\usepackage[utf8]{inputenc}
\usepackage[T1]{fontenc}
\usepackage[all,cmtip]{xy}

\usepackage{graphicx}
\usepackage[colorlinks=true, allcolors=blue]{hyperref}
\usepackage[dvips]{epsfig}

\usepackage[dvips]{graphicx}
\usepackage[svgnames]{xcolor}
\definecolor{linkcolor}{HTML}{e88d67} 

\usepackage{ulem}
\usepackage{color}

\newtheorem{theorem}{Theorem}[section]

\newtheorem{lemma}[theorem]{Lemma}
\newtheorem{definition}[theorem]{Definition}
\newtheorem{proposition}[theorem]{Proposition}
\newtheorem{claim}{Claim}
\newtheorem{remark}{Remark}

\newcommand{\RR}{\mathbb{R}}

\newcommand{\ZZ}{\mathbb{Z}}

\newcommand{\pp}{\partial}

\newcommand{\cA}{\mathcal A}

\newcommand{\hg}{H_{N}^{\max}}
\newcommand{\GT}{\Gamma}
\newcommand{\OO}{\Omega}

\newcommand{\GB}{(\Gamma,\kappa)}
\newcommand{\hn}{H_N}
\newcommand{\avhn}{\overline{H}_N}

\DeclarePairedDelimiter\floor{\lfloor}{\rfloor}

\newcommand\restr[2]{{
  \left.\kern-\nulldelimiterspace 
  #1 
  \vphantom{\big|} 
  \right|_{#2} 
  }}
\newcommand{\bb}{\begin{equation}}
\newcommand{\ee}{\end{equation}}

\newcommand{\norm}[1]{\left\lVert#1\right\rVert_{\infty}^{\Omega}}
\newcommand{\norma}[1]{\left\lVert#1\right\rVert_{\infty}}

\usepackage{pst-node}
\usepackage{tikz-cd}

\begin{document}

\title{A variational principle for domino tilings of multiply-connected domains}

\author[Nikolai~Kuchumov]{ \ Nikolai~Kuchumov$^\star$}


\thanks{\thinspace ${\hspace{-.45ex}}^\star$ Sorbonne Universit\'e, CNRS, Laboratoire de Probabilit\'es Statistique et Modelisation,
LPSM, UMR 8001, F-75005 Paris, France.
\hskip.06cm
Email:
\hskip.06cm
\texttt{kuchumov@lpsm.paris}}

\vskip.2cm

\begin{abstract}
We study uniformly random domino tilings of a multiply-connected domain with a multivalued height function according to the usual definition. We consider it as a function on the universal covering space of the domain that makes it a well-defined function. It allows us to prove that as the domain grows, normalized height
functions converge in probability to a deterministic continuous function in two regimes. The first regime covers all domino tilings of the domain, for which we also prove a convergence in probability of the height change. In the second regime, we consider domino tilings with a given height change $R_N$ that converges to a fixed asymptotic height change $r$. 
\end{abstract}

\maketitle

\vskip.25cm
\section{Introduction}

Counting the number of domino tilings of a region such as the chess board is a classical problem in combinatorics. However, how does a \textit{typical domino tiling of a given region look like}? The first attempt answering this question for  a simply-connected region is given by the variational principle in \cite{CKP}.~The authors encode a domino tiling $D$ by the \textit{height function} $H_D$, and show that, after appropriate scaling, the height function of a typical domino tiling is close to a solution of the variational problem. However, for a multiply-conneted region this description typically fails, and $H_D$ is a multivalued function.
The goal of this paper is to define the height function on a multiply-connected domain, and extend the variational principle for domino tilings of a multiply-connected domain.

\begin{figure}[h!]
    \begin{minipage}[t]{0.30\textwidth}
    \includegraphics[width=0.6\linewidth]{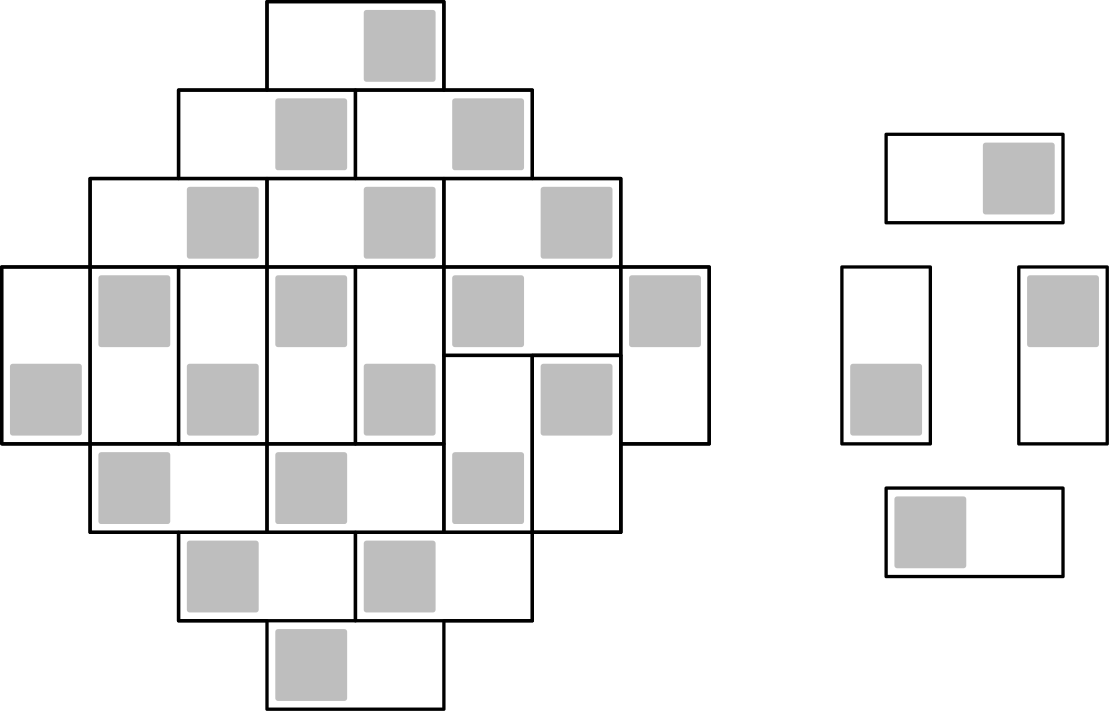}
    \end{minipage}
    \begin{minipage}[t]{0.35\textwidth}
    \raggedleft     \raggedleft
    \includegraphics[width=0.6\linewidth]{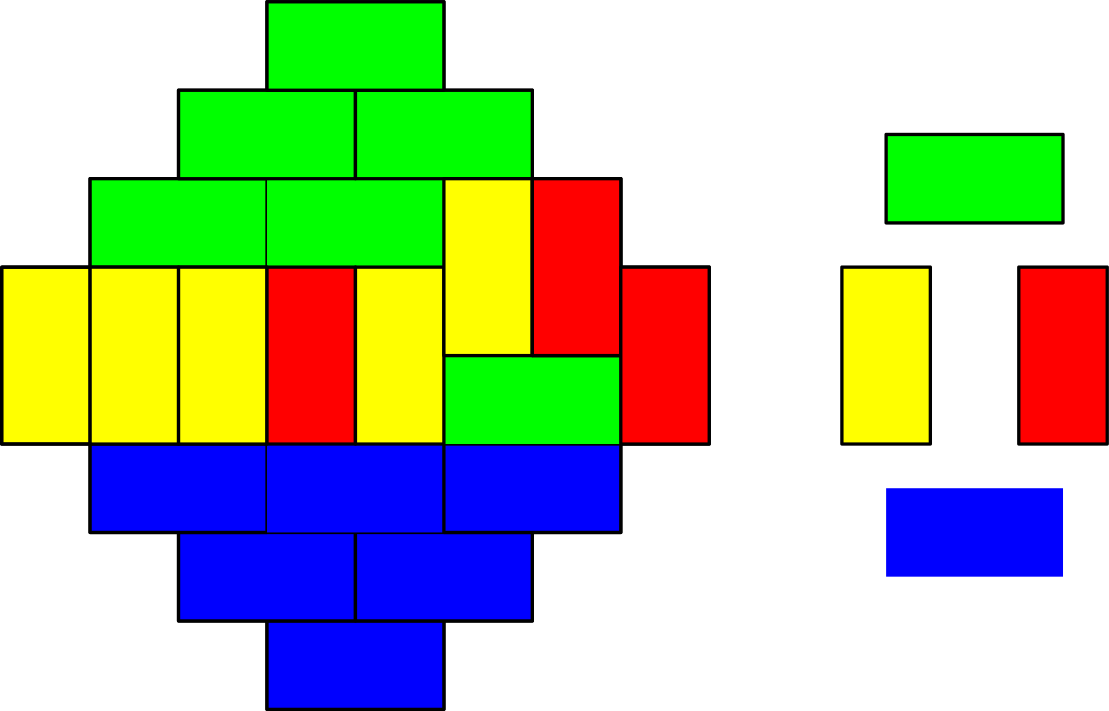}
    \end{minipage}
    \caption{The same domino tiling of Aztec diamond $AD_4$ with two graphical representations and four types of dominoes w.r.t. their orientation.}
    \label{azteccolor}
\end{figure}

Historically, one of the first such theorems regarding the law of large numbers for~geometrical objects coming from combinatorics is the famous Vershik--Kerov--Logan--Shepp limit shape for the Young diagrams \cite{VK,LS}.~After that, there began a widespread development in the area starting with the Arctic Circle Theorem in \cite{JPS}.~The authors analyzed typical domino tilings of a growing sequence of lattice domains, so-called Aztec diamonds $AD_N$(see \hyperref[azteccolor]{Figure \ref{azteccolor}}).~It was shown that for large $N$, a uniformly random domino tiling of $AD_N$ forms two kinds of regions separated by the circle inscribed to the normalized domain, the unit square rotated by $45^{\circ}$.~Statistics of dominoes in the inner region remains non-trivial in the limit as $N\to\infty$, all four types of dominoes w.r.t. their orientation have a positive density. However, regions outside the circle exhibit a deterministic statistics, each of them is tiled by a fixed type of domino.

Soon after,~H.Cohn, R.Kenyon and J.Propp proved that this phenomenon holds for a generic simply-connected domain in \cite{CKP}.~Let us give a glance on their main theorem with necessary details. 

Let $\Gamma\subset\ZZ^2$ be a finite, connected region with a fixed chessboard coloring, and viewed as a subset of $\RR^2$ with set of vertices $V(\Gamma)$. A domino tiling of $\Gamma$ is a covering of it without gaps or overlaps with dominoes (i.e., $1\times 2$ or $2\times 1$ rectangles)
We call a region $\Gamma$ tileable if it admits a domino tiling. Equip the set of domino tilings of $\Gamma$ with the uniform measure $\mathbb{P}^{\Gamma}$.
Then, define the boundary of $\Gamma$, $\pp\Gamma:=\{ p\in\Gamma| p\sim \ZZ^2\setminus\Gamma \}$, where $\sim$ means graph adjacency on $\ZZ^2$.

The main technical tool used in \cite{CKP} is an encoding of a domino tiling $D$ by a height function, $H^{\Gamma}_D: V(\Gamma)\to\ZZ$, defined on vertices of $\Gamma$ fixed by $H_{D}^{\Gamma}(0)=0$, as the integral of a natural flow associated to the domino tiling $D$, for details see \hyperref[defh]{ Definition \ref{defh}}. 

They consider a sequence of lattice domains $\Gamma_N\subset\frac{1}{N}\ZZ^2$ that approximates a connected, simply-connected compact set $\Omega\subset\RR^2$. Furthermore, it is assumed that the boundary condition $B_N:=\restr{H_D}{\pp\Gamma_N}$ after normalization by $N^{-1}$ converges to a deterministic continuous Lipschitz function~$\mathfrak{b}: \pp\Omega\to\RR$.

Let $\mathfrak{h}^{\star}$ be the unique maximizer over space of Lipschitz functions with boundary condition $\mathfrak{b}$ of \textit{the surface tension functional} $\mathcal{F}: h\mapsto\iint_{\Omega}\sigma(\pp_x h,\pp_y h)dxdy$, which we postpone defining until \hyperref[Surf]{Section \ref{Surf}}.~The main theorem of \cite{CKP} states that the normalized logarithm of the number of domino tilings of $\Gamma_N$ with a boundary condition $B_N$, $Z(\Gamma_N,B_N)$, has the following asymptotic behavior as $N\to\infty$.

\begin{theorem}[Cohn Kenyon Propp]
\bb
N^{-2}\log{Z(\Gamma_N,B_N)}\xrightarrow{N\to\infty} \mathcal{F}\left( \mathfrak{h}^{\star}\right).
\ee 
Furthermore, the normalized height functions $\frac{1}{N} H_D$ converge uniformly in probability to $\mathfrak{h}^{\star}$ as $N\to\infty$.
\label{orig}
\end{theorem}

However, this theorem does not cover domino tilings of a multiply-connected domain such as the modified Aztec diamond, see \hyperref[fancy]{Figure~\ref{fancy}} and \hyperref[modi]{Section~\ref{modi}}. The first reason is that a height function $H_D^{\Gamma}$ defined by \hyperref[defh]{The Local rule} becomes a multivalued function, that is it gains a non-trivial increment (monodromy) $M(\delta)$ going along a loop $\delta\in\pi_1(\Omega)$, see \hyperref[ring]{Figure~\ref{ring}}. The monodromy $M(\delta)$ is fixed by the domain and $\delta$ for all domino tilings.

The second reason is that in a multiply-connected domain, the boundary condition $B$ usually depends on $D$, see \hyperref[ring]{Figure~\ref{ring}} and \hyperref[island]{Figure~\ref{island}}. More precisely, it means that boundary height functions $B_{D_1}$ and $B_{D_2}$ may differ by a multiple of four on a connected boundary component. Thus, boundary conditions a priori do not converge to a continuous function $\mathfrak{
b}$. Therefore, the variational problem may be ill-defined.

Despite these difficulties, several results on domino tilings of multiply-connected domains exist. We make a review of these works in \hyperref[results]{Final remarks}. 

\begin{figure}[h!]
    \centering
    \includegraphics[width=0.4\linewidth]{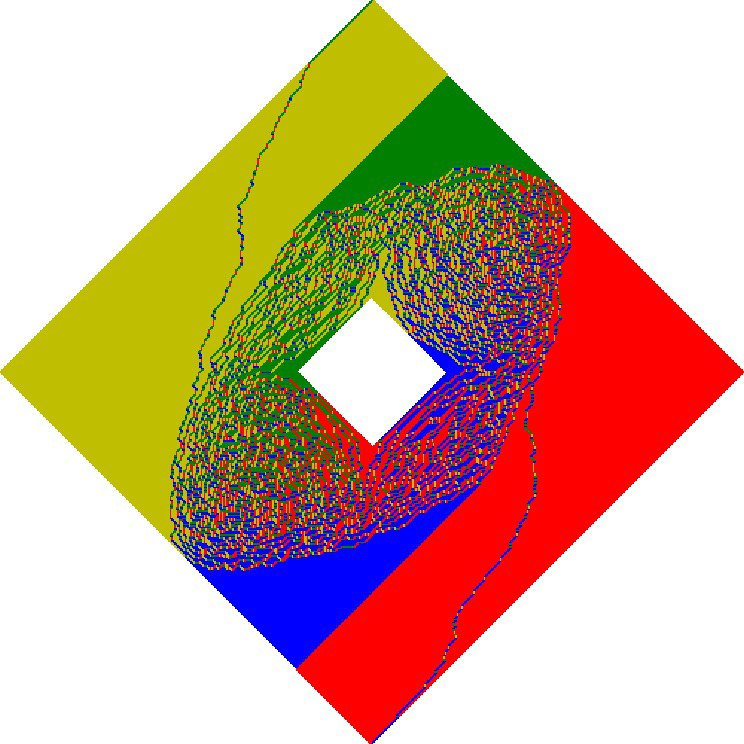}
    \caption{Computer simulation of a random domino tiling of $\mathcal{AD}_{500}$ with height change $R=300$ and monodromy $M=400$. For further details, see \hyperref[modi]{Section~\ref{modi}}.}
    \label{fancy}
\end{figure}

\begin{figure}[h!]
    \centering
    \includegraphics[width=0.3\textwidth]{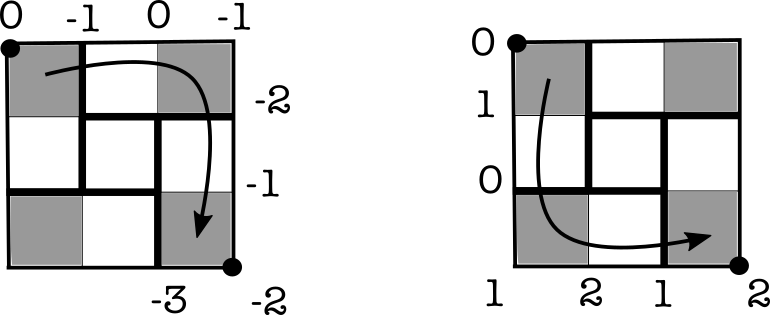}
    \caption{Height function $H(p,\gamma)$ and $H(p,\gamma^{\prime})$ with $M(\gamma^{\prime}\gamma^{-1})=4$, for details see \hyperref[defh]{Definition \ref{defh}}.}
    \label{ring}
\end{figure}

We formulate a generalization of \hyperref[orig]{Theorem~\ref{orig}} for a multiply-connected domain in \hyperref[lim]{Theorem~\ref{lim}}. To the best of our knowledge, it is the first generic result for random domino tilings of a multiply-connected domain that works with a height function with a monodromy. 

Instead of defining the height function by \hyperref[defh]{The Local Rule}, we modify it so that it depends not only on the point $p\in\Gamma$, but also on a path $\gamma$ from a fixed reference point $p_0$ to $p$. This dependence keeps track of the topology and makes the height function a well-defined function.~After going around a hole, not only the value of the height function changes, but also the path in the argument. This situation is very similar to the definition of the complex logarithm map, for details see discussion in \hyperref[complexlog]{Remark \ref{complexlog}}.
We also keep track of relative heights between different points on each connected boundary component. Assuming that the value of the height function at the point $p_0$ is zero, we end up with a collection $R_D:=\{H^{\Gamma}(p_i)_D\}_{i=1}^g$ that consists of $g$ values at points $\{p_i\}_{i=1}^g$ on every connected boundary component, we also need to fix the set of paths modulo continuous transformation $\{\gamma_i\}_{i=1}^g$ along which we compute $H^{\Gamma}_D(p_i)$. Call such a collection of numbers a \textit{height change} of domino tiling $D$.


Define a similar $g$-tuple of real numbers $r=\{r_i\}_{i=1}^g$ for a continuous function, that is, fix its value at a fixed point on each connected boundary component with the set of paths $\{\gamma_i\}_{i=1}^g$ continuous transformation. Before the main theorem, let $\mathfrak{b}$ be an extendable boundary condition on the universal cover of $\Omega$ such that there exists such a function $\mathfrak{h}:\Tilde{\Omega}\mapsto\RR$ that  $\restr{\mathfrak{h}}{\pp\Tilde{\Omega}}=\mathfrak{b}$. 

Then, let $\mathfrak{h}^{\star}$ be the maximizer of $\mathcal{F}$ over the space of Lipschitz functions on $\Tilde{\Omega}$ with boundary condition $\mathfrak{b}$ with an arbitrary height change.
Our main result is the following,

\begin{theorem}
The normalized logarithm of the number of domino tilings of $\Gamma_N$ divided by the area of $\Omega$ has the following asymptotic behavior as $N\to\infty$,
\bb
N^{-2}\log Z(\Gamma_N,B_N)\xrightarrow{N\to\infty} \mathcal{F}(\mathfrak{h}^{\star}).
\ee
\label{lim}
Moreover, $\frac{1}{N} H_D^{\Gamma_N}\overset{\mathbb{P}_N}{\to}\mathfrak{h}^{\star}$ uniformly in probability, and the normalized height change of domino tilings converge to the height change of $\mathfrak{h}^{\star}$, $\frac{1}{N}R_N\overset{\mathbb{P}_N}{\to} r^{\star}$ as $N\to\infty$.
\end{theorem}

Note that in $Z(\Gamma_N,B_N)$ we count domino tilings with arbitrary height change. The version of this theorem for a fixed asymptotic height change $r$ also holds, see \hyperref[lim2]{Corollary \ref{lim2}}. Furthermore, for a simply-connected $\Omega$ we recover \hyperref[orig]{Theorem \ref{orig}} since $\Tilde{\Omega}$ becomes $\Omega$ itself, and it has only one connected boundary component on which the height function is fixed. What is important is that for a multiply-connected $\Omega$, $\Tilde{\Omega}$ is not a compact space unlike the situation for simply-connected domain.

\subsection*{Acknowledgements}
The study was funded within the framework of the HSE University Basic Research Program, and the project DIMERS, ANR-18-CE40-0033.
We would like to thank Senya Shlosman and C\'edric Boutilier for carefully reading the manuscript, valuable remarks and fruitful discussions. We also thank Vadim Gorin for his comments on the work and sharing his insights and ideas. We are very grateful to Alexei Gordeev and Kirill Simonov for their help with computer simulation of random domino tilings. Finally, we would like to thank KOMETA space, especially Osya Gordon, for their hospitality and useful discussions at the beginning of the work.

 \subsection*{Structure of the paper}
 
\begin{itemize}
    \item In \hyperref[sect2]{Section \ref{sect2}} we discuss topological notations: universal cover, fundamental domain of it, and quasiperiodic functions together with analogy between the height function and the complex logarithm map.
    \item \hyperref[sect3]{Section \ref{sect3}} is devoted to the definition of the height function on universal covering space and the bijection between domino tilings and height functions.

    \item In \hyperref[sect4]{Section \ref{sect4}} we prove several properties of the height functions: we give a formula for the maximal extension with given boundary height function, criterion of extension of height functions in \hyperref[crith]{Proposition \ref{crith}}, and define the space of asymptotic height functions with a given, and an arbitrary) asymptotic height change. Also, in \hyperref[comp]{Theorem \ref{comp}} we prove compactness of both spaces. We also prove \hyperref[density]{Theorem \ref{density}}, which states that discrete height functions approximate asymptotic height functions and vice versa. 
    \item In \hyperref[sect5]{Section \ref{sect5}} we prove the concentration inequality for the uniform measure on domino tilings of multiply-connected domains in \hyperref[Concentration_lemma]{Lemma \ref{Concentration_lemma}}. Then, and we recall the exact formula for function $\sigma$ in (\ref{Surfacedef}).
    
    \item In \hyperref[sect6]{Section \ref{sect6}} we state and prove the variational principle for arbitrary height change, \hyperref[The_limit]{Theorem \ref{The_limit}},
    and for a fixed asymptotic height change, \hyperref[lim2]{Corollary \ref{lim2}}
    \item In \hyperref[sect7]{Section \ref{sect7}}, we prove existence and uniqueness of the maximizer for the variational problem in \hyperref[exist]{Proposition \ref{exist}}.

    \item \hyperref[sect8]{Section \ref{sect8}} is devoted to \hyperref[Functional]{Theorem \ref{Functional}}.

    \item The last \hyperref[sect9]{Section \ref{sect9}} contains final remarks.
    
\end{itemize}

\section{Topological notation}

In this section we recall and fix basic topological notations, which we will use later, for details see \cite{H,DC,JM}.
\label{sect2}

\subsection{Fundamental domain}

First, let $\Omega$ be a planar domain, that is a compact subset of $\RR^2$ with a piecewise smooth boundary $\pp\Omega$, which consists of $g+1$ connected components, $\pp\Omega=\bigsqcup_{i=0}^g \pp\Omega_i$, where $\pp\Omega_0$ is the external boundary. 

Recall that the universal covering space $\Tilde{\Omega}$ is a simply-connected set, and can be realized as a set of pairs $(x,\gamma_x)$ of points $x\in\Omega$ with paths $\gamma_x$ that connect $x$ to a fixed base point $x_0\in\pp \Omega$. Moreover, one can notice a natural action of fundamental group of $\Omega$ on $\Tilde{\Omega}$. Recall that having two paths modulo continuous transformations, $\gamma_1$ connecting $p_1$ to $p_2$ and $\gamma_{2}$ connecting $p_2$ to $p_3$, one can obtain a third path $\gamma_3:=\gamma_2 \circ \gamma_1$. This path goes first from $p_1$ to $p_2$ and then from $p_2$ to $p_3$ called a \textit{concatenation} or \textit{composition} of paths $\gamma_1$ with $\gamma_2$, which is defined also modulo continuous transformations. If we take $p_2=p_3$ so that $\gamma_2$ is a loop, we obtain an action of $\pi_1(\Omega)$ on points of $\Tilde{\Omega}$ that maps a point $(p,\gamma)$ to $(p,\delta\cdot\gamma)$.
Denote the action of $\delta\in\pi_1(\Omega)$ on a subset $\mathcal{A}\subset \Tilde{\Omega}$ by $\delta\cdot\mathcal{A}$. It is also worth noting that $\Omega=\Tilde{\Omega}\slash  \pi_1(\Omega)$.

\subsubsection{Definition of fundamental domain}
Let us define a fundamental domain $\mathcal{D}(\Omega)$ as follows, 
\begin{itemize}
    \item $\mathcal{D}(\Omega)\subset\Tilde{\Omega}$ is a closed set
    \item $
\Tilde{\Omega}=\bigcup_{\delta\in\pi_1(\Omega)}{\delta\cdot\mathcal{D}(\Omega)}$, where $(\delta_1\cdot\mathcal{D}(\Omega))\cap(\delta_2\cdot\mathcal{D}(\Omega))$ has no interior for $\delta_1\neq\delta_2$.
\label{fund_def}
\end{itemize}
Below, we give one realization of $\mathcal{D}(\Omega)$ in \hyperref[fund]{Construction \ref{fund}} that we use later.

Let us explain the definition of $\mathcal{D}(A)$ on a ring $A$ with a radius $1$ of the internal circle and radius $2$ of the external circle. It can be seen in polar coordinates as 
\bb
A=\{ (\ell\cos \phi,\ell \sin \phi )|\ell\in[1,2], \phi\in \RR\slash{2\pi\ZZ}\},
\ee
or as $\{(\ell,\phi)|\ell \in [1,2], \phi\in\RR\slash{2\pi\ZZ}\}$ in different coordinates. Then, the universal cover can be seen as 
\bb
\Tilde{A}=\{(\ell, \phi)|\ell\in[1,2], \phi\in\RR\}.
\ee

The corresponding fundamental domain is $\mathcal{D}(A)=\{ (\ell, \phi )|\phi\in [ 0,2\pi ], \ell\in[1,2]\}$. Here the fundamental group $\pi_1(A)\simeq\mathbb{Z}$ acts by shifts $(\ell,\phi)\mapsto(\ell,\phi + 2\pi k), k\in \ZZ$, where the integer $k$ represents the winding number of a loop. Note that here we have $g=1$. \label{rring}

\subsubsection{Construction of fundamental domain.}
\label{constuction_fund}
The example above can be generalized to a construction of $\mathcal{D}(\Omega)$ for a generic multiply-connected domain $\Omega$.
First, pick $g$ smooth curves $\{\gamma_i\}^g_{i=1}$ on $\Omega$ that connect $\pp\Omega_0$ with all other connected boundary components $\{\pp\Omega_i\}_{i=1}^g$ and avoid self-intersections such that
after the cut along curves $\{\gamma_i\}_{i=1}^g$ the resulting domain $\mathcal{D}_0(\Omega)$ is connected and simply-connected.
Also note that $\mathcal{D}_0(\Omega)$ can be naturally viewed as a subset of $\Tilde{\Omega}$.

Then, we take the closure of $D_0(\Omega)$ in $\Tilde{\Omega}$ and obtain $\mathcal{D}(\Omega)=\overline{\mathcal{D}_0(\Omega)}$. It can be proved that at the end we obtain a fundamental domain, for details, see Section 1.1, 1.3 in \cite{H}.
\label{fund}
Note that the boundary of a fundamental domain consists of an extra $2g$ boundary pieces in addition to $\pp\Omega$, $\pp\mathcal{D}(\Omega)=\pp\Omega \cup \bigcup_{i=1}^{2g}\upsilon_{i}$, where $\upsilon_i$ and $\upsilon_{i+g}$ are the result of cutting along $\gamma_i$.~One can also note that $\pp\mathcal{D}(\Omega)$ has one connected boundary component.

\subsection{Decomposition of paths}

Each path $\gamma_x$ from $x_0$ to $x$ can be decomposed after a continuous transformation into a product $\gamma_x=\hat \delta \cdot \gamma_x^{\prime}$ of a path $\gamma^{\prime}_x$ conditioned to stay inside $\mathcal{D}(\Omega)$, and a loop $\hat \delta\in\pi_1(\Omega)$. It is easy to see that such a path $\gamma^{\prime}_x$ is unique modulo continuous transformations after choosing $\mathcal{D}(\Omega)$. Also note that the action of $\delta\in\pi_1(\Omega)$ on $\Tilde{\Omega}$ maps $\hat \delta$ by multiplying by $\delta: \hat \delta\mapsto \delta \cdot \hat \delta$, while leaving $\gamma^{\prime}_x$ unchanged.
This allows us to look at a point of $\Tilde{\Omega}$ as a pair $(\hat x,\delta)$ of a point $x\in\mathcal{D}(\Omega)$ and a loop $\hat\delta$, which is changing after going around a loop $\delta^{\prime}$ as $(x,\hat\delta)\mapsto(x,\delta^{\prime}\cdot\hat\delta)$. Also, it is worth comparing with \hyperref[fund_def]{Definition \ref{fund_def}}.

\subsection{Quasiperiodic functions}
Throughout the paper, all functions on $\Tilde{\Omega}$ are assumed to be \textit{quasi-periodic}. These functions are similar to the complex logarithm map, which is almost a well-defined map, but it gets the increment $2\pi i$ after one turn around zero in counter-clockwise direction.
\label{quasiperiodic}

Suppose that we have a monodromy map $\mathfrak{m}$, that is a map $\mathfrak{m}:\pi_1(\Omega)\to\RR$ that  $\mathfrak{m}(\delta^{\prime}\cdot\delta)=\mathfrak{m}(\delta^{\prime})+\mathfrak{m}(\delta)$.
Then, $f:\Tilde{\Omega}\to\RR$ is a quasiperiodic function with monodromy $\mathfrak{m}$ if after going along a loop $\delta$ on $\Omega$, it changes as 
\bb
f(x,\delta\cdot\hat \delta)=f(x,\hat \delta)+\mathfrak{m}(\delta),
\ee
or in other words upon substitution $\delta=\delta^{\prime}\cdot\hat \delta^{-1}$,
\bb
f(x,\hat \delta)=f(x,\hat \delta^{\prime})-\mathfrak{m}(\delta^{\prime}_x\cdot \hat \delta^{-1}).
\ee 
The monodromy $\mathfrak{m}$ keeps track of the possible ambiguity of $f$ after going along a loop. Denote the space of quasiperiodic functions for a given monodromy $\mathfrak{m}$ by $\mathcal{H}(\Omega,\mathfrak{m})$. 

Two examples of quasiperiodic functions are the complex logarithm map $\log z$ defined on a punctured complex plane (for details see \hyperref[complexlog]{Subsection \ref{complexlog}}), and the function $(\phi,\ell)\mapsto \phi+ \ell \cos \phi $ defined on the ring $A$ from \hyperref[rring]{Example \ref{rring}}. Since $\pi_1(A)\simeq \ZZ$, each $\hat \delta$ corresponds to an integer, which enumerates “landings of a stairwell” see \hyperref[universal]{Figure \ref{universal}}.

The action of $\pi_1(\Omega)$ on $\Tilde{\Omega}$ naturally extends to an action of $\pi_1(\Omega)$ on quasi-periodic functions on $\mathcal{A}\subset\Tilde{\Omega}$. That is, for a quasi-periodic function $f$ and a $\delta\in\pi_1(\Omega)$ define $\delta \cdot f(p,\hat \delta):=f(p,\delta \cdot \hat \delta)= f(p, \hat \delta)+\mathfrak{m}(\delta)$.

Another property of quasi-periodic functions is that the difference between two functions $f, f^{\prime}$ from $\mathcal{H}(\Omega,\mathfrak{m})$, $f-f^{\prime}$, is a well-defined function on $\Omega$ since $f-f^{\prime}$ has monodromy $0$.

\begin{lemma}
Suppose $f,f^{\prime} \in \mathcal{H}(\Omega,\mathfrak{m})$.
Then, functions defined by $x\mapsto f(x)-f^{\prime}(x)$ and $x\mapsto\nabla f(x)$ are well-defined as functions on $\Omega$.
\label{well}
\end{lemma}

\begin{proof}

Indeed, $f$ and $f^{\prime}$ gain the same increment $\mathfrak{m}(\delta)$ going around the loop $\delta$, therefore $f-f^{\prime}$ has the trivial monodromy, and thus, it is a well-defined function on $\Omega$. Moreover, if $f$ is differentiable, then its gradient by definition is also a well-defined map on $\Omega$.
    
\end{proof}

\begin{remark}
Since the gradient of continuous(discrete) height function is a well-defined map, we could in principle formulate all the statements in the language of continuous/discrete gradients. Discrete gradients parametrize domino tilings without any ambiguity as naively defined height functions.
\begin{figure}
    \centering
    \includegraphics[width=0.3\linewidth]{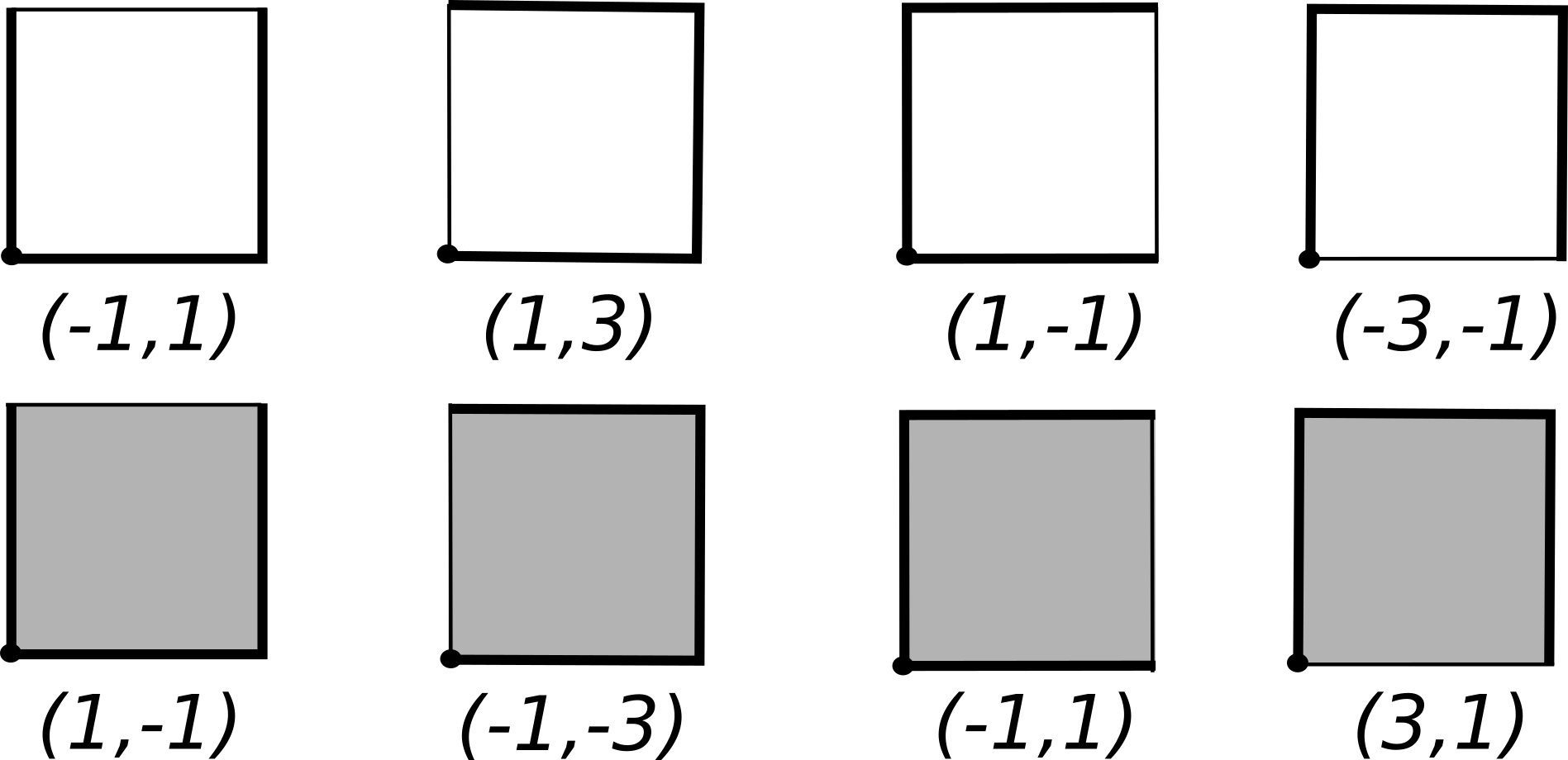}
    \caption{Parametrization of domino tilings by the discrete slopes.}
    \label{slope}
\end{figure}
Probably, this is a possible alternative to our approach. However, it would require a more sophisticated technique i.e., topology of the space of continuous gradients as in a recent work\cite{WS}.
\end{remark}

An important fact about quasiperiodic functions is that their behavior is determined by their values on $\mathcal{D}(\Omega)$. It can be formulated as the following proposition.
\begin{proposition}
Suppose we have two quasi-periodic functions $f,f^{\prime}\in\mathcal{H}(\Omega,\mathfrak{m})$ that coincide on $\mathcal{D}(\Omega)$.
Then, $f=f^{\prime}$ on $\Tilde{\Omega}$.
\label{unique}
\end{proposition}

\begin{proof}
By definition of the fundamental domain, we can express the universal covering space $\Tilde{\Omega}$ as the union of fundamental domains,
\bb
\Tilde{\Omega}=\bigcup_{\delta\in\pi_1(\Omega)}{\delta\cdot\mathcal{D}(\Omega)}.
\label{union}
\ee
Now, the proof is straightforward, we need to use two facts. The first fact is the assumption that the values of $f,f^{\prime}$ coincide on $\mathcal{D}(\Omega)$. The second one is that both functions have the same monodromy data $\mathfrak{m}(\delta)$ that gives agreement of their values on $\delta\cdot\mathcal{D}(\Omega)$.
\end{proof}

\hyperref[unique]{Proposition \ref{unique}} allows recovering a quasi-periodic function $f$ defined on $\Tilde{\Omega}$ uniquely from its restriction to $\mathcal{D}(\Omega)$  if we know the monodromy $\mathfrak{m}$.

Call the subspace of functions on $\mathcal{D}(\Omega)$ obtained by restriction of functions from $\mathcal{H}(\Omega,\mathfrak{m})$
\bb
\mathcal{H}^{\prime}\left(\mathcal{D}(\Omega),\mathfrak{m}\right):=\{f: f=\restr{g}{\mathcal{D}(\Omega)} \text{ for some } g\in \mathcal{H}(\Omega,\mathfrak{m}) \}.
\ee

\begin{proposition}
Functional spaces $\mathcal{H}^{\prime}(\mathcal{D}(\Omega),\mathfrak{m})$ and $\mathcal{H}(\Omega,\mathfrak{m})$ are isomorphic with each other. In other words, each function $f^{\prime}\in \mathcal{H}^{\prime}(\Omega,\mathfrak{m})$ extends to a function $f\in\mathcal{H}(\Omega,\mathfrak{m})$ uniquely. 
\end{proposition}
\begin{proof}
We have the projection from $\mathcal{H}(\Omega,\mathfrak{m})$ to $\mathcal{H}^{\prime}(\Omega,\mathfrak{m})$ that is surjective by definition of $\mathcal{H}^{\prime}(\Omega,\mathfrak{m})$. In order to construct the map in the opposite direction suppose that $\hat{f}:=\restr{f}{\mathcal{D}(\Omega)}$, so we know the value $\hat{f}(x)=f(x,\mathcal{I})$ for the trivial loop $\mathcal{I}$. Then, $f(x,\delta)=\hat{f}(x,\mathcal{I})+\mathfrak{m}(\delta\cdot\mathcal{I})=\hat{f}(x,\mathcal{I})+\mathfrak{m}(\delta)$. Therefore, once we know that $\hat{f}$ can be extended to a function on $\Tilde{\Omega}$ with monodromy $\mathfrak{m}$, we can recover its values on the whole $\Tilde{\Omega}$ from the values on $\mathcal{D}(\Omega)$.

\label{iso}
\end{proof}

\subsection{Boundary condition of quasi-periodic functions.}

Let us take a point $x_i$ on a connected boundary component $\pp\Tilde{\Omega}_i$ for each $0 \leq i \leq g$, and let $f\in\mathcal{H}(\Omega,\mathfrak{m})$. 
Boundary conditions of $f$ is a set of functions $\mathfrak{b}_i:\pp\Tilde{\Omega}_i\to\RR$ on each connected boundary component, such that $\restr{f}{\pp\Tilde{\Omega}_i}(x,\delta)-f(x_i,\delta) =\mathfrak{b}_{i}(x,\delta)$ for $x\in\pp\mathcal{D}(\Omega)_i$. This condition determines $f$ on $\pp\Tilde{\Omega}_i$ up to an additive constant that can be fixed by an extra condition $f(x_i,\delta)=r_i$ for an $r_i\in\RR$. We also assume that $f(x_0,\mathcal{I})=0$

Once we think of $f$ as a surface, the value $r_i:=f(x_i,\delta)-f(x_0,\delta)$ represents the relative height of $\pp\Tilde{\Omega}_i$ compared to $\pp\Tilde{\Omega}_0$. Call $\{r_i\}_{i=1}^g$ the height change. Then, there are two ways to fix the boundary condition of $f$, the first way is to take $g+1$ functions $\{\mathfrak{b}_i\}_{i=0}^g$ on each connected boundary component $\pp\Tilde{\Omega}_i$ and fix $\restr{f}{\pp\\Tilde{Omega}_i}(x,\delta)-f(x_i,\delta)=\mathfrak{b}_i(x,\delta)$. This way we fix f up to an additive shift, denote the space of such quasi-periodic functions on $\Tilde{\Omega}$ with monodromy $\mathfrak{m}$ $\mathcal{H}(\Omega,\mathfrak{m},\{\mathfrak{b}_i\})$. The second option is to further fix value $f(x_i,\mathcal{I})=r_i$ for $r_i\in\RR$, call the space of such functions $\mathcal{H}(\Omega,\mathfrak{m},\{\mathfrak{b}_i\},r_i)$.

\section{Height function on universal covering space of discrete domains}

This section is devoted to definitions of height functions and asymptotic height functions.

Our first goal is the definition of the universal cover of a lattice domain. Then, we define the height function and construct a bijection between the set of height functions and the set of domino tilings of a multiply-connected lattice domain $\Gamma$.

\label{sect3}

\subsection{Universal cover of a lattice domain}
Suppose that we have a multiply-connected domain $\Omega$ with a lattice domain $\Gamma\subset\Omega$. Let us define the universal cover of $\Gamma$ denoted by $\Tilde{\Gamma}$ as a subset of $\Tilde{\Omega}$ as follows.
Recall that the expression of the universal cover of $\Omega$,
$\Tilde{\Omega}=\bigcup_{\delta\in\pi_1(\Omega)}{\delta\cdot\mathcal{D}(\Omega)}$.

Since $\Gamma$ is a subset of $\Omega$, define the fundamental lattice domain $\mathcal{D}(\Gamma)$ to be the lattice domain remaining after performing the cuts $\gamma_i$ from \hyperref[constuction_fund]{construction \ref{constuction_fund}}, where the cuts are continuously changed so that a path $\gamma_i$ consists of edges of $\Gamma$.

Then, we can define the universal cover of a lattice domain $\Gamma$ denoted by $\Tilde{\Gamma}$ similarly to \hyperref[fund_def]{Definition \ref{fund_def}},
\bb
\Tilde{\Gamma}:=\bigcup_{\delta\in\pi_1(\Omega)}{\delta\cdot\mathcal{D}(\Gamma)}.
\ee
It is easy to see that $\Tilde{\Gamma}$ is a subset of $\Tilde{\Omega}$.

\subsection{Height function}

First, recall the usual definition of the height function for a simply-connected domain $\Gamma$.
Define the boundary of $\Gamma$, $\pp\Gamma:=\{ p\in\Gamma| p\sim \ZZ^2\setminus\Gamma \}$, where $\sim$ means graph adjacency on $\ZZ^2$. Later we will assume that $\Gamma$ approximate a continuous domain $\Omega$, so that $\pp\Gamma$ will approximate $\pp\Gamma$.

The height function is defined by \textit{the Local Rule} as follows. 
\begin{definition}
   A function $H_D^{\Gamma}: V(\Gamma)\to\ZZ$ is called a height function if for a fixed $p_0\in\pp\Gamma$ the following holds, 
\begin{enumerate}
    \item Set the value of $H_D^{\Gamma}(p_0):=0$ for all $D$ and a fixed point $p_0\in\pp\Gamma$.
    \item If the edge $v:=(p_1,p_2)$ has a black square on its left, then $H_D^{\Gamma}(p_2)$ equals $H^{\Gamma}_D(p_1) + 1$ if $v$ does not cross a domino in $D$ and $H_D^{\Gamma}(p_1) - 3$ otherwise.
    \label{defh}
\end{enumerate}
\end{definition}
By \cite{T}, $H_D^{\Gamma}$ is a well-defined for a simply-connected $\Gamma$.~However, in a multiply-connected region, $H_D^{\Gamma}$ might get a non-zero increment going along a loop around a hole, see \hyperref[ring]{Figure \ref{ring}}. Therefore, we need to modify this definition so that it will be still a well-defined map.

Let us give an intrinsic definition of the height function $H^{\Gamma}$ defined on a connected, multiply-connected lattice domain $\Gamma$.
Here we follow our conventions from \hyperref[quasiperiodic]{Subsection \ref{quasiperiodic}}.

We pick once point $p_i\in\Gamma$ on each connected boundary component $\pp\Gamma_i$, where $\pp\Gamma_0$ is the external connected boundary component of $\Gamma$. Fix also a set of paths $\{\gamma_i^{\prime}\}_{i=1}^g$ from $p_0$ to each $p_i$ and a path $\gamma$ from $p_0$ to a point $p\in\Gamma$.

Also, let $\{R_i\}_{i=0}^{g}$ be a collection of integers with $R_0=0$.

Let us mark a single edge on every lattice square of $\Gamma$. We assume that this edge is not a boundary edge. And suppose that $M^{\Gamma}$ is a monodromy, that is $M^{\Gamma}:\pi_1(\Omega)\to 4 \ZZ$ and satisfies $M^{\Gamma}(\delta \cdot \delta^{\prime})=M^{\Gamma}(\delta)+M^{\Gamma}(\delta^{\prime})$.
Let us proceed to definitions of height function and its boundary condition,

\begin{definition}
Call a family of integer functions $B_i^{\Gamma}:\Tilde{\pp\Gamma_i} \to\ZZ$ the boundary heights if they satisfy the following conditions
\begin{itemize}
\item the value of $B_i^{\Gamma}$ increases by $1$ counterclockwise along the edges of any black square except the marked edge on it.
\item the value of $B_i^{\Gamma}$ increases by $1$  clockwise along all edges of any white square except the marked edge on it.
\item going along a loop $\delta_i \in \pi _1(\Omega)$ $B^{\Gamma}_i(v,\delta_v)$ changes as follows:
        \bb
     B_i^{\Gamma}(v,\delta_i\cdot\delta)=B_i(v,\delta)+M^{\Gamma}(\delta_i).
     \ee
\end{itemize}
\end{definition}

Then, we can define the height function with given boundary heights.

\begin{definition}
A function $H^{\Gamma}:\Tilde{\Gamma}\to \ZZ$ is a height function with monodromy $M$, set of reference points $\{p_i\}_{i=0}^g$, boundary heights $B_i$ and height change $\{R_i\}_{i=0}^g$ if the following holds,
\begin{itemize}
    \item $\restr{H}{\pp\Gamma_i}(p,\delta)-H(p_i,\delta)=B_i(p,\delta)$ for $p\in\pp\Tilde{\Gamma}_i$.
    \item $H(p_i,\gamma_i^{\prime})=R_i$.
    \item the value of $H$ increases by $1$  counterclockwise along all edges of any black square except the marked edge on it.
    \item the value of $H$ increases by $1$  clockwise along all edges of any white square except the marked edge on it.
    \item going along a loop $\delta \in \pi _1(\Omega)$ $H(v,\delta_v)$ changes as follows:
        \bb
     H(v,\delta\cdot\delta_v)=H(v,\delta_v)+M^{\Gamma}(\delta).
     \ee
\end{itemize}
\label{hdef}
\end{definition}

We think of a marked edge as a dual object to a domino, in other words to a dimer. In fact, the marked edge cuts a domino into two unit squares. Thus, a boundary of a domino is formed by non-marked edges (edges with increment $\pm 1$). 

Note that a different choice of points $\{p_i\}_{i=1}^g$ changes $\{R_i\}_{i=1}^g$ by an additive shift. For instance, if we change a point $x_i$ to an adjacent point $p_i^{\prime}\in\pp\Gamma$, the height change component $R_i$ will change by $\pm 1$ depending on the value $H(p_i^\prime)$.

The height function defined this way on the ring is, in fact, a well-defined function only on the corresponding “Riemannian surface” that can be schematically viewed on \hyperref[universal]{Figure \ref{universal}}.
\begin{figure}
    \centering
    \includegraphics[width=0.3\linewidth]{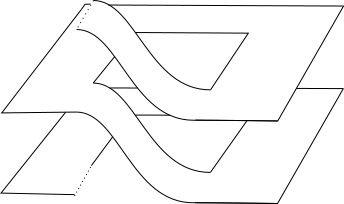}
    \caption{“Riemannian Surface” of a ring.}
    \label{universal}
\end{figure}

\subsection{Analogy with complex analysis}
In order to fully grasp the definition of the height function on a multiply-connected domain, it is useful to keep in mind the definition of $\log z$ for complex $z$ through the analytic continuation along a path.
\label{complexlog}
Recall that one can define $\log z$ for $z\in \mathbb{C}\setminus \{0\}$ by fixing the value $\log (1)=0$, and then by the direct analytic continuation along a path $\gamma$ that connects point $1$ with point $z$ via $\log_{\gamma}(z)=\int_{\gamma}\frac{d\zeta}{\zeta}$. 
However, two such logarithms $\log_{\gamma}z$ and $\log_{\gamma^{\prime}} z$ defined this way may differ by an integer multiply of $2\pi i$. Therefore, one needs to choose a cut to make $\log_{\gamma}(z)$ a single-valued function with a branch-cut with a tradeoff of discontinuity at the branch-cut.~There is also a way to make $\log z$ a continuous function defined on the corresponding so-called “Riemann surface” \cite{Lang}.~A point of this surface, known as the universal covering space, can be seen as a pair of a point $z\in\mathbb{C}\setminus \{ 0\}$ with a path connecting $z$ to point $1$ modulo continuous transformations in $\mathbb{C}\setminus 0$.

In \hyperref[monod]{Definition~\ref{monod}} we define the height function $H_D$ of domino tiling $D$ of a ring similarly to $\log z$. In order to choose a branch, we make a cut along $\RR_{>0}$ and fix a value of $H_D$ at a point $p_0$. Then, the residue of $\log_\gamma(z)$ is analogous to monodromy of $H_D$. Globally, the height function $H_D$ is also a well-defined function only on the corresponding “Riemann surface” $\Tilde{A}$. Later on, we generalize this setup to a more complicated domain.

In the following discussion, by a path we mean an oriented lattice path on $\ZZ^2$, that is a sequence of oriented edges. Let also $\gamma, \gamma^{\prime}$ be two paths modulo continuous transformations that connect point $p_0$ with point $p$. Note in mind that $\gamma^{-1} \gamma^{\prime}$ is a loop. Also recall that $I$ is the constant map of the point $p_0$ to itself. See a summary of this analogy between $H$ and $\log_{\gamma}$ in \hyperref[tabl]{Table \ref{tabl}}.

\begin{table}[h]
    \centering
    \begin{tabular}{c|c}
       $\log_{\gamma} z$  & $H_D(p,\gamma)$ \\
       Residue of $\frac{1}{z}$ at $z=0$ & Monodromy of $H$ \\
       
       $\log_{I}(1)=0$ & $H_D(p_0,I)=0$ \\
       $\log_{\gamma}(z)=\int_{\gamma} \frac{d\zeta}{\zeta}$ & \hyperref[defh]{The Local rule} \\
       $\log_{\gamma}z=\log_{\gamma^{\prime}}z+\int_{\gamma^{-1}\gamma^{\prime}}\frac{d\zeta}{\zeta}$ & $H_D(p,\gamma)=H_D(p,\gamma^{\prime})+M(\gamma^{-1}\gamma^{\prime})$ \\
       Branch-cut along $\RR_{>0}$  & Cut along a path connecting different
       connected boundary components.
    \end{tabular}
    \caption{Analogy between height function $H_D(p,\gamma)$ and $\log_{\gamma} z$}
    \label{tabl}
\end{table}

Let us show correctness of the definition of height function $H$. That is increments of $H$ along any two paths $\gamma,\gamma^{\prime}$ with the same starting point $p$ and ending point $p^{\prime}$ coincide once the paths can be continuously transformed one into the other.

\begin{lemma}
The increment of $H$ along the path $\gamma$ equals to its increment along the path $\gamma^{\prime}$ if these paths can be continuously transformed one into another.
\bb
H(p^{\prime},\gamma)=H(p^{\prime},\gamma^{\prime})
\ee
\label{homotopy}
\end{lemma}
\begin{proof}

It is an easy exercise on topology to show that a lattice domain $\Gamma_{\gamma}^{\gamma^{\prime}}$ obtained between $\gamma$ and $\gamma^{\prime}$ is simply-connected, for details see Chapter 1 in \cite{H}. Therefore, the height function on $\Gamma_{\gamma}^{\gamma^{\prime}}$ is well-defined. Thus, the value of $H$ at $p^{\prime}$ can be computed along both paths with the same result.
\end{proof}

Monodromy $M^{\Gamma}(\delta)$ is, in fact, fixed by $\pp\Gamma$ and one can write a formula for it along a single loop $\delta$.

Let $e$ be an edge of $\Gamma$ and define $\Delta_e$ to be $+1$ if $e$ has black square on its left and $-1$ otherwise. 
A single loop $\gamma$ divides $\ZZ^2$ into two connected components, $\Gamma^{\prime}$, the finite component, and $\Gamma^{\prime\prime}$, the infinite component. Denote the number of black squares in $\Gamma^{\prime}$ by $\mathcal{B}(\Gamma^{\prime})$ and the number of white square in $\Gamma^{\prime}$ by $\mathcal{W}(\Gamma^{\prime})$.

\begin{lemma}
The monodromy $M(\delta)$ of a height function $H$ on a lattice domain $\Gamma$ does not depend on $H$, and it can be computed by the following formula,

\bb
M(\delta)=\sum_{e\in\delta}\Delta_e= \mathcal{B}(\Gamma^{\prime})-\mathcal{W}(\Gamma^{\prime}).
\ee
\label{form}
\end{lemma}

\begin{proof}

The loop $\delta$ can be continuously shrunk to the connected boundary component $\pp\Gamma^{\prime}$. The increment of $H$ going along $\pp\Gamma^{\prime}$ can be computed using \hyperref[defh]{The Local Rule}, and, since there are no marked edges on $\pp\Gamma^{\prime}$, the value of $H$ changes by $\Delta_e=\pm 1$ only. Therefore, the monodromy can be written as follows, 

\bb
M(\delta)=\sum_{e\in\pp\Gamma^{\prime}} \Delta_{e}.
\ee

~The monodromy $M(\delta)$ divided by $4$ equals to the difference between $\mathcal{B}(\Gamma^{\prime})$ and $\mathcal{W}(\Gamma^{\prime})$.
\bb
M(\delta)=\sum_{e\in\pp\Gamma^{\prime}} \Delta_{e}=4\left(\mathcal{B}(\Gamma^{\prime})-\mathcal{W}(\Gamma^{\prime})\right)
\label{mon_fromula}
\ee

This formula follows from the Stokes theorem applied to the height function, see section 2 in \cite{K} and \cite{T}. 
\end{proof}
The formula from \hyperref[form]{Lemma \ref{form}} extends to an arbitrary loop as follows.
Suppose we have a loop $\delta\in\pi_1(\Omega)$ and a set of generators of $\pi_1(\Omega)), \{\alpha_i\}_{i=1}$. The loop $\delta$ can be decomposed into generators, $\delta=\prod_{i=1}\alpha_i^{a_i}$. From (\ref{mon_fromula}), we have a formula for monodromy around each $\alpha_i$. Suppose without loss of generality that $\alpha_i$ can be continuously shrunk to $\pp\Gamma_i$. Then, we can compute the monodromy $M(\delta)$ as follows,
\bb
M(\delta)=M\left(\prod_{i=1}\alpha_i^{a_i}\right)=\sum_{i=1}a_i\cdot M(\alpha_i)=~
\sum_{i=1}a_i\cdot M(\pp\Gamma_i)=4\sum_{i=1}a_i\cdot(\mathcal{B}(\Gamma_i^{\prime})-\mathcal{W}(\Gamma_i^{\prime})).
\ee


\subsection{Height function of a domino tiling}

A typical example of a height function can be obtained starting with a domino tiling $D$. We modify the usual definition of height function corresponding to a domino tilling from \cite{T,CKP,F} such that it becomes a well-defined function for a domino tiling of a multiply-connected domain.
\begin{definition}

Let $D$ be a domino tiling of a multiply-connected lattice domain $\Gamma$. Then, $H_D:\Gamma \to\ZZ$ is a height function of a domino tiling $D$ if it satisfies the following properties.

\begin{itemize}
\item Fix the value of $H_D$ at $(p_i,\gamma_i)$, $H_D(p_i,\gamma_i)=R_i$.
\item if an edge $(p,v)$, does not belong to any domino in $D$ (i.e. intersects) then $H_D(v,\gamma)=H_D(p,\gamma)+1$ if $(p, v)$ has a black square on the left, and $H_D(v,\gamma)= H_D(p,\gamma)-3$ otherwise.
\item going along a loop $\delta \in \pi _1(\Omega,p_0)$ $H_D(v,\gamma)$ changes as follows:
\bb
H_D(v,\delta\gamma_v)=H_D(v,\gamma_v)+M(\delta).
\ee
\label{monod}
\end{itemize}

\end{definition}

Now we are ready to give a proof of the following theorem,
\begin{theorem}
There is a bijection between the set of height functions on $\Gamma$ and domino tilings of it.
\label{bijection}
\end{theorem}

\begin{proof}
The height function of a domino tiling defined ~\hyperref[monod]{Defenition \ref{monod}} obviously satisfies conditions of \hyperref[hdef]{Definition \ref{hdef}}, the marked edges are the edges that cross the dominoes.

In the other direction we need to assign a domino tiling to a height function $H$. Note that $H$ changes by $-3$ counterclockwise along the marked edge of $\mathcal{S}_b$, which is because there are three non-marked edges of $\mathcal{S}_b$ with increments $+1$. Therefore, there exists a unique edge on every lattice square with increment $3$. Moreover, each marked edge is, automatically, the marked edge of two adjacent squares of different colors. Thus, we can declare that the domino covers these two squares. This way we can cover by dominoes all the squares of $\Gamma$ and obtain a domino tiling.

One can further notice that marked edges are dual to the perfect matching on the dual graph $\Gamma^{\star}$ in language of the dimer model.
\end{proof}

\section{Proofs of properties of height functions}
\label{sect4}
We have two kinds of height functions, discrete height functions, which arise from domino tilings, and asymptotic height functions, which are natural limiting objects for sequences of height functions. Both height functions and their continuous counterparts share many properties. Thus, it is worth recalling a basic construction for a usual Lipschitz function $f$ \cite{V}.

Let $M$ be a metric space with the distance $d(\cdot,\cdot)$. Suppose that $X\subset M$ is a compact subset of $M$ and $f: X \to \RR$ is a Lipschitz function.
Then, $f$ can be extended to a Lipschitz function on $M$ with the same Lipschitz constant by the following formula:
\bb
\hat f(x)=\min_{y \in X}(f(y)+d(x,y))
\label{ext}
\ee
$\hat f$ is also the maximal extension of $f$ that is for any extension $g$ of $f$ we have $g\leq \hat f$. Furthermore, one can write a similar formula for the minimal extension, 
\bb
\check{f}(x):=\max_{y \in X}(f(y)+d(x,y)).
\ee
For quasiperiodic functions on $\Omega$ a similar formula holds, let $f:\pp\Tilde{\Omega}\to\RR$ be a function with monodromy $\mathfrak{m}$ and 
$(x,\gamma_x),(y,\gamma_y) \in\mathcal{D}(\Omega)$. Then, 
\bb
\hat f(x,\gamma_x)=\inf_{(y,\gamma_y) \in \pp\mathcal{D}(\Omega)}(f(y,\gamma_y)+d((x,\gamma_x),(y,\gamma_y)).
\ee

This gives a well-defined map on $\mathcal{D}(\Omega)$ since $\pp\Omega$ is a compact set, and thus the value $\hat f(x,\gamma_x)$. Then, we need to check that it also gives a well-defined quasi-periodic function defined on $\Tilde{\Omega}$ with the same monodromy data as $f$. Let us move $\hat f(x,\gamma_x)$ around a loop $\delta$. By definition, the expression $f(y,\gamma_x)+d((x,\gamma_x),(y,\gamma_y))$ changes to $f(y,\gamma\cdot\delta)+d((x,\gamma_x),(y,\gamma_y))+\mathfrak{m}(\delta)$ since $d((x,\delta\cdot\gamma_x),(y,\delta \cdot \gamma_y))=d((x,\gamma_x),(y,\gamma_y))$, and the change is independent of $y$. Therefore, we have the desired transformation, so $\hat f$ is a quasi-periodic function.

\subsection{Extension of a boundary height function.}
In the following discussion, we need functions that satisfy only the second condition of \hyperref[monod]{Definition~\ref{monod}}. Let us call them partial height functions. To prove that a partial height function $\eta$ is an actual height function, we need to check that $\eta$ has the right boundary conditions.
Recall an analog of the Lipschitz condition for height functions from \cite{CEP}: let $\Gamma$ be a lattice domain with a height function $H$ defined on it.
Let $\beta^{\Gamma}(p_1,p_2)$ be the length of a minimal path $\pi$ joining points $p_1$ and $p_2$ inside $\Gamma$ such that every edge of $\pi$ (oriented from $p_1$ towards $p_2$) has a black square on its left. Also, it is not hard to check that $\beta^{\Gamma}(p_1,p_2)$ satisfy the second condition of \hyperref[monod]{Definition~\ref{monod}} as function of $p_1$ with a fixed $p_2$. Thus, $\beta^{\Gamma}(p_1,p_2)$ is the maximal increment of the height function between $p_1$ and $p_2$($\beta^{\Gamma}$ is only increasing by $1$ on the path).

Then, for every two points $p_1,p_2 \in \Gamma$,
\bb
H(p_1)\leq H(p_2)\pm\beta^{\Gamma}(p_1,p_2) 
\label{Lipschitz_simply}
\ee
Let us also set $\beta(p,p):=0$.
Note that this condition implies a constraint for height change of a height function, $-\beta^{\Gamma}(x_i,x_0)\leq R_i\leq\beta^{\Gamma}(x_i,x_0)$.

\begin{figure}[h!]
    \centering
    \includegraphics[width=0.3\linewidth]{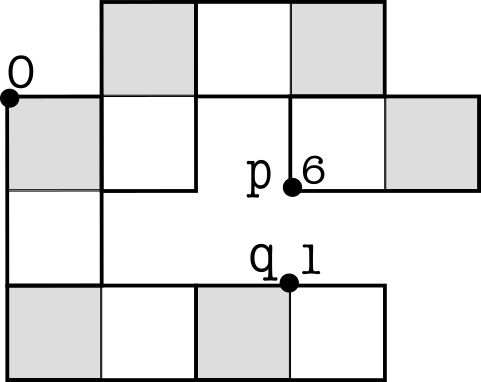}
    \caption{Example of a domino tiling of a concave domain, where the height function does not satisfy the Lipschitz condition for points $p$ and $q$ with $\beta$ defined by (\ref{beta_form}).}
    \label{concave}
\end{figure}

Call Condition (\ref{lattice_lipschitz}) \textit{the lattice Lipschitz condition}. For points $p_1,p_2$ at distance $d$ in sup-norm and a convex $\Gamma\subset\ZZ^2$, $\beta^{\Gamma}(p_1,p_2)\leq 2 d(p_1,p_2) + 1$ \cite{CEP}. Also, one can write an exact expression for $\beta(x,y)$ for such a $\Gamma$ taken from \cite{F} in the form of lemma $2.1$ from \cite{PST}.

Suppose $x=(a,b),y=(a+i,b+j)\in\ZZ^2$ and set $\kappa(i,j)=i-j~mod~2$. If $a-b=0~mod~2$, one sets
\bb
\beta (x,y)= \left\{
    \begin{split}
     2\| y-x \|_{\infty}+ \kappa (i,j),  \qquad \text{ if } i \geq j, \\
     2\| y-x \|_{\infty}- \kappa (i,j),  \qquad \text{ if } i < j.    \end{split}
  \right.
\ee
If $a-b=1 \mod 2$, then
\begin{equation*}
  \beta (x,y)=\left\{
    \begin{split}
     2\| y-x \|_{\infty}- \kappa (i,j), \qquad  \text{ if } i \geq j, \\
    2\| y-x \|_{\infty}+ \kappa (i,j), \qquad  \text{ if } i < j.    \end{split}
  \right.
  \label{beta_form}
\end{equation*}

We need several properties of the height function on a multiply-connected lattice domain that are crucial for the proofs on a simply-connected lattice domain. Let us recall them \cite{PST}, later we extend them to a multiply-connected lattice domain.

These properties are analogous to the basic properties of a usual Lipschitz function. Moreover, we will show that the height function is a discrete analog of the Lipschitz function in a sense that it approximates continuous Lipschitz functions after a suitable normalization on a large lattice domain, see \hyperref[density]{Theorem \ref{density}}.
In order to show it, we need the following lemma,

\begin{lemma}
Assume that $H_D$ and $H^{\prime}$ are height functions. Assume also that $p$ is a vertex of $\Gamma$.
Then, $H(p)=H^{\prime}(p)~mod~4$
\label{mod4}
\end{lemma}
\begin{proof}
By its definition, two height functions coincide at the reference point $(p_0,\gamma_0)$. Then values at an adjacent point $p^{\prime}$ may be either $H$ or $H\pm 4$ according to the definition of height function \hyperref[monod]{Definition~\ref{monod}}. Or in other words, the discrete gradient of the difference $H-H^{\prime}$ along any edge always takes its values in $\{-4,0,4\}$.

This way we can obtain the proof using induction with respect to the length of a path connecting the point $p$ with the reference point $p_0$.
\end{proof}

Another important property of the set of height functions is that it has additional operations, which turn this set into a lattice. Namely, we can take pointwise maximum or pointwise minimum of two such functions and obtain another element from this set. For a proof for a simply-connected domain, see discussion in section $1.4$ of \cite{CKP}. Since we use a modified height function, we need to show this property for the functions defined by \hyperref[defh]{Definition~\ref{defh}}.
Suppose that $\{p_i\}_{i=0}^{i=g}\in\pp\Gamma$ points on each connected boundary component of $\Gamma$ and $H$ and $H^{\prime}$ are two height functions on $\Gamma$ with height changes $\{R_i\}_{i=0}^g$ and $\{R^{\prime}_i\}_{i=0}^{g}$. Then,

\begin{lemma}
Pointwise maximum of two height functions $H, H^{\prime}$ defined on a lattice domain $\Gamma$ is again a height function $\check{H}$ defined as
\bb
\check{H}(p):=\max\{H(p),H^{\prime}(p)\}.
\ee
Further, height change of $\check{H}$ is maximum between height changes of $H, H^{\prime}$,
\bb
\check{R}_i=\max\{R_i, R^{\prime}_i\}.
\ee
\label{point}
\end{lemma}

\begin{proof}

We need to show that $\check{H}$ satisfies four conditions from \hyperref[defh]{Definition~\ref{defh}}.
To show the second and the third properties we can use the same arguments as in \cite{CKP,PST}. That is supposing that $H(v)\neq H^{\prime}(v)$ for a point $v\in\Gamma$ (if there is no such a point, it would imply $\check{H}=H=H^{\prime}$). Then, $H(v)$ must differ from $H^{\prime}(v)$ by a multiple of $4$ by \hyperref[mod4]{Lemma \ref{mod4}}. Without loss of generality, suppose that $H(v)=H^{\prime}(v)+4\ell$ for $\ell\in\mathbb{N}^{\star}$. Together with the definition of height function, this implies that at all points $u$ adjacent to $v$, $H(u) \leq H^{\prime}(u)$. Therefore, $\check{H}$ at points $v$ and $u$ coincides with $H$. Repeating this argument we obtain the desired.

For the fourth condition note that monodromy of a height function is fixed by the domain by \hyperref[form]{Lemma \ref{form}}.

Finally, let us prove the first condition. Note that the height change of $\check{H}$ equals the values of $\check{H}$ at the reference points on the $\pp\Gamma$, $\check{R}=\{\check{R}_i\}_{i=0}^g=\check{H}(p_i)_{i=0}^g$. These values are maxima between the values of $H$ and $H^{\prime}$ at reference points. Thus, height change of $\check{H}$ equals to the maximum between height change of $H$ and $H^{\prime}$, $\check{R}_i=\max\{R_i,R_i^{\prime}\}$.

Here we assumed that two height functions are defined on $\Tilde{Gamma}$ and coincide at the reference point $(p_0,\gamma_0)$. However, the proof also works with the assumption that they coincide modulo $4$ that allows us to use \hyperref[mod4]{Lemma \ref{mod4}}. Further, monodromies of all height functions are the same since monodromy is determined by the lattice domain by \hyperref[mon_formula]{Lemma \ref{mon_fromula}}.

\end{proof}
Moreover, the same works for pointwise minimum of the height functions, pointwise minimum of height functions $H$ and $H^{\prime}$ is, again, the height function, denote it $\Tilde{H}$,
\bb
\Tilde{H}(p)=\min\{H(p),H^{\prime}(p)\}.
\ee

Height function on a multiply-connected domain satisfy a refined lattice Lipschitz condition,
\bb
 H(x,\hat \delta)-H(y,\delta_{y}) \leq \beta^{\Gamma}(x,y)+M(\hat \delta\cdot\delta_y^{-1}).
\label{lattice_lipschitz}
\ee
This condition can be obtained by applying \hyperref[Lipschitz_simply]{Condition \ref{Lipschitz_simply}} to a height function on$\mathcal{C(\Omega)}$. This constraint allows us to formulate a criterion of extension of boundary values to a height function as follows,

\begin{proposition}
Let $\Gamma$ be a multiply-connected domain with a fixed height function $H$ with monodromy $M$ on a subset $\bar{\Gamma}\subset\Gamma$.
Then, $H$ admits an extension to a height function on $\Gamma$ if the following inequality holds for all points $x$, $y$ of $\bar{\Gamma}$,

\bb
 H(x,\hat \delta)-H(y,\delta_{y}) \leq \beta^{\Gamma}(x,y)+M(\hat \delta\cdot\delta_y^{-1})
\label{Height_Lipschitz}
\ee
\label{crith}
\end{proposition}

\begin{proof}

Let us take a family of height functions on $\Gamma$, $\{ \check{H}_y: x \mapsto H(y)+\beta(x,y) \}$ indexed by the points $x$ of $\bar{\Gamma}$.
The pointwise minimum of two partial height function from the family is again a  partial height function by \hyperref[point]{Lemma \ref{point}} since $\check{H}_y(x)$ satisfies the local rule as a function of $x$. Then, taking the pointwise minimum over the whole family, we get a height function on $\Gamma$

\bb
H^{\max}(x,\gamma_x):=\min_{y \in  \bar{\Gamma}}(H^{\prime}(y,\gamma_y))
\ee

We need to show that $H^{\max}$ agrees with $H$ on $\bar{\Gamma}$.

Let $x,y,v\in\bar{\Gamma}$, and let us take a function $H^{\prime}_x:y\mapsto H(y)+\beta^{\Gamma}(x,y)$. It is clear that $H(x)\leq H(y)+\beta^{\Gamma}(x,y)$ holds for every $y$ since it is the lattice Lipschitz condition for $H$, therefore we have $H(v)\leq H^{\max}(v)$. Inequality $H^{\max}(v)\leq H(v)$ also holds because for $y=x$ we have the equality since $\beta^{\Gamma}(x,x)=0$.

The height change and monodromy of $H^{\max}$ are given by the boundary condition $H$.

\label{minmax}
\end{proof}
Note that due to (\ref{Height_Lipschitz}) $H^{\max}$ is the maximal extension of $H$ to a height function on $\Gamma$. The minimal extension of $H$ can be constructed by almost the same way as the maximal. Let us define the minimal extension $H^{min}$,
\bb
H^{\min}(x,\gamma):=\max_{y \in \bar{\Gamma}}(H(y,\gamma)-\beta(x,y)).
\ee

\subsection{Asymptotic height function}


Here we define an asymptotic height function as a continuous counterpart to 
the discrete height function.
We saw that discrete height functions satisfy the lattice Lipschitz condition. Therefore, it is natural to expect that in a scaling limit they approximate continuous Lipschitz functions. However, as the example on \hyperref[concave]{Figure \ref{concave}} suggests we need to consider functions that are 2-Lipschitz in intrinsic sup norm, which can be defined by the following formula

\bb
\norm{x-y}:=\inf_{\gamma}\int_{0}^{1} \norma{\gamma^{\prime}(t)}dt.
\ee
Here the curve $\gamma$ connects $x$ to $y$ by as a path inside $\Omega$, $\gamma(0)=y, \gamma(1)=x$. Moreover, let us use such a path of the minimal length w.r.t. the intrinsic distance on $\Omega$.

Let us note that theorems for the standard sup norm such as Arzela Ascoli theorem and Rademacher's theorem hold for intrinsic norm as well, as one can check using the same arguments as for standard sup norm locally (i.e. by using partition of unity with convex support).

Fix a set of points $\{(x_i,\gamma)\in\mathcal{D}(\Omega)|x_i\in\pp\Omega_i\}$ and a monodromy data, i.e., a map $\mathfrak{m}:\pi_1\to\RR$ such that $\mathfrak{m}(\gamma\cdot \gamma^{\prime})=\mathfrak{m}(\gamma)+\mathfrak{m}(\gamma^{\prime})$ and $\mathfrak{m}(\gamma^{-1})=-\mathfrak{m}(\gamma)$.
Also, let $r:=\{r_i\}_{i=0}^{g}$ be a sequence of real numbers and denote point $z=(z_1,z_2) \in \Omega$. 

Also define the Newton polygon, which is the set of allowed slopes for the height function arising from domino tilings, $\mathcal{N}:=\{(x,y)\in\RR^2| |x|+|y|\leq 2\}$ \cite{CKP,KOS}

A function $\mathfrak{h}$ is an asymptotic height function with height change $r=\{r_i\}_{i=1}^g$ and monodromy $\mathfrak{m}$ if the following holds,

\begin{itemize}
    \item $\restr{\mathfrak{h}}{\pp\Omega_i}(x)-\mathfrak{h}(x_i)=\mathfrak{b}_i(x)$
    \item $\mathfrak{h}(x_i,\gamma_i)=r_i,$
    \item $\mathfrak{h}$ is $2$-Lipschitz function with respect to the intrinsic sup norm on $\Omega$, that is
    
    $|\mathfrak{h}(x,\delta)-\mathfrak{h}(y,\delta^{\prime})|\leq 2 \norm{x-y}+\mathfrak{m}(\delta^{-1}\delta^{\prime})$ and $\nabla \mathfrak{h}:\Omega\to\mathcal{N}$
    \item after going around a loop $\delta^{\prime}$, values of $\mathfrak{h}$ changes as follows:\\
    $\mathfrak{h}(x,\delta^{\prime}\cdot\delta)=\mathfrak{h}(x,\delta)+\mathfrak{m}(\delta^{\prime}).$
\end{itemize}

It is important to mention that numbers $r_i$ for asymptotic height function take their values in intervals due to the Lipschitz constraint, $-2\norm{x_i -x_0}\leq r_i\leq 2\norm{x_i -x_0}$.

Let $\mathscr{H}(\Omega,\mathfrak{m},\mathfrak{b}_r)$ be the space of asymptotic height functions with monodromy $\mathfrak{m}$ and boundary condition $\mathfrak{b}_r$. Also define a union of $\mathscr{H}(\Omega,\mathfrak{m},\mathfrak{b}_r)$ over all possible height changes $r$,  $\mathscr{H}(\Omega,\mathfrak{m},\mathfrak{b}):=\bigcup_r \mathscr{H}(\Omega,\mathfrak{m},\mathfrak{b}_r)$.

\begin{theorem}
The spaces $\mathscr{H}(\Omega,\mathfrak{b}, \mathfrak{m})$ and $\mathscr{H}(\Omega,\mathfrak{b}_r, \mathfrak{m}, r)$
are compact spaces with respect to the intrinsic sup norm.
\label{comp}
\end{theorem}
\begin{proof}
The idea of the proof is to show compactness of the space of functions on the fundamental domain that will lead us to compactness of $\mathscr{H}(\Omega,\mathfrak{b}, \mathfrak{m})$. Then, $\mathscr{H}(\Omega,\mathfrak{b}_r, \mathfrak{m}, r)$ is compact as a closed subset of $\mathscr{H}(\Omega,\mathfrak{b}, \mathfrak{m})$. 

Recall that by \hyperref[unique]{Proposition  \ref{unique}}, the behavior of these functions is determined by their values on $\mathcal{D}(\Omega)$. Thus, it is sufficient to show compactness of $\mathscr{H}(\mathcal{D}(\Omega),\mathfrak{b})$. In order to show it, one can apply Arzela-Ascoli theorem. The first requirement, the existence of a uniform bound for the functions $f\in\mathscr{H}(\mathcal{D}(\Omega),\mathfrak{b})$, is satisfied because of the boundary condition $\mathfrak{b}$ that is fixed on $\pp\Omega_0$. The equicontinuous follows directly from the Lipschitz condition.
Then, these function with a given monodromy data $\mathfrak{m}$ form a closed subspace of $\mathscr{H}(\mathcal{D}(\Omega),\mathfrak{b})$. Thus, it is a compact space. The same works for the subspace with a given height change $r$ and a fixed monodromy data $\mathfrak{m}$, $\mathscr{H}(\Omega,\mathfrak{b}_r, \mathfrak{m}, r)$. The latter is, again, a compact space as a closed subset of $\mathscr{H}(\Omega,\mathfrak{b}, \mathfrak{m})$.
\end{proof}

Note that the gradient of an asymptotic height function is well-defined on $\OO$. Further, pointwise difference between two quasi-periodic functions with the same monodromy is, again, a well-defined function on $\OO$

\subsection{Approximations of a domain}

Assume that $\Omega\subset\RR^2$ is a domain and $\Gamma_N$ tends to $\Omega$ as $N\to\infty$ with respect to the Hausdorff distance $d_{H}$, $d_{H}(X,Y)=\inf\{\varepsilon \geq 0:  X\subseteq Y_{\varepsilon }\ \text{ and } \ Y\subseteq X_{\varepsilon }\}$, where  $X_{\varepsilon }$ is $\varepsilon$-neighborhood of $X$, $d_H(\Omega,\Omega_{N}) \to 0,$ as $N \to \infty$.

Now, let $\Gamma_N\ \subset \frac{1}{N}\ZZ^2\cap \Omega$ be a sequence of lattice domains with normalized boundary conditions $\{B_N\}$. Let us follow ideas from \cite{CKP,KOS,AG} and call a sequence of regions with boundary conditions $(\GT_{N},B_N)$ \textit{an approximation} of $(\Omega,\mathfrak{b})$  if

\begin{enumerate}
\item $\GT_{N} \subset \frac{1}{N} \ZZ^2 \cap \Omega$, where $\frac{1}{N} \ZZ^2$ is $\ZZ^2$ with mesh $\frac{1}{N}$.
\item each $\GT_{N}$ admits at least one domino tiling with normalized boundary condition $B_{N}^{R_N}$ for each non-trivial $R_N=\{R_N^i\}_{i=1}^g$.
\item for every admissible asymptotic height change $r$, there exists a sequence of admissible normalized height changes $\{R_N\}$ that converges to $r$, $R_N^i\to r^i$ as $N\to\infty$.
\item  $\Gamma$ tends to $\OO$ with respect to the Hausdorff distance $d_H$, $d_H(\Omega,\Gamma) \to 0,$ as $N \to \infty$.
\item
Further, $|B_N^{R_N}(x_N)- \mathfrak{b}^r (x)| \leq O(N^{-1})$ for the standard Euclid distance for sufficiently large $N$ with $x_N\in\pp\Gamma, x\in \pp\Omega$ such that $|x_N-x|\leq O(N^{-1})$ (the existence of such points is guaranteed by the previous assumption). 
\end{enumerate}

Furthermore, note that the convergence of boundary conditions means that the discrete monodromy data $\{\frac{1}{N}M_i\}$ converges to the continuous $\{m_i\}$.

These conditions are modifications of conditions of Theorem $1.1$ from \cite{CKP}, which are needed for multiply-connected domains. Convergence of the boundary conditions implies convergence of discrete monodromy $M_N(\delta)$ to its continuous counterpart $\mathfrak{m}(\delta)$.

\subsection{Convergence of the maximal extensions}

In this section we prove convergence of discrete norm $\beta^{\Gamma}(x,y)$ to continuous intrinsic norm $\norm{x-y}$.
Firstly, we need to normalize $\beta$,
\bb
\beta^{\Gamma_N}_N:=\frac{\beta^{\Gamma}}{N}.
\ee
For $\Gamma=\ZZ^2$ and convex domains we have an exact formula for $\beta$ (\ref{beta_form}), which tells that $\beta(x,y)$ approximates sup norm up to a factor $\pm 1/N$. For other domains it still holds.

Let $\Gamma_N$ be an approximation of a domain with a boundary condition $\Omega$. Then,

\begin{proposition}
$\exists C>0$ such that $\forall p,q \in\Gamma_N$ viewed as points of $\Omega$, the following holds  
\bb    
    |\beta^{\Gamma_N}_N(p,q) - \norm{p-q}|\leq\frac{C}{N}
\ee 
\end{proposition}

\begin{proof}
    Let us take a path $\gamma$ connecting $x$ to $y$ that minimizes integral from $\norm{x-y}$. Then, cover $\gamma$ by a family of open balls $\mathcal{B}_i\subset\Omega$ that lie strictly inside $\Omega$, $\gamma\subset \bigcup_{i=1}^C \mathcal{B}_i$, let $C$ be their number. On each ball we have the convergence of $\beta_N(x,y)$ to $\norm{x-y}$ by formula (\ref{beta_form}), which might give an error $\frac{1}{N}$ on each $\mathcal{B}_i$. Thus, we have the desired estimate on $\Omega$ up to an error of $\frac{C}{N}$.
\end{proof}

~Let domain $\OO$ be a domain with a~fixed boundary condition $\mathfrak{b}:\pp\Tilde{\OO}\to\RR$. Further, let ($\Gamma_N,B_N)$ be an approximation of $(\OO,\mathfrak{b})$, and let $H_N^{\max}$ be the maximal extension of $B_N$. Then, $H_N^{\max}$ approximates the maximal extension of $\mathfrak{b}$ as $N\to\infty$,

\begin{proposition}
Then for sufficiently large $N$ and $x\in\Gamma_N$ viewed as point of $\Omega$,
\bb
|H_N^{\max}(x) - h^{\max}(x)| \leq O(N^{-1})
\ee
\label{maxext}
\end{proposition}
Note that there is a lattice analog of the formula of the maximal extension, where we take the minimum is taken over points of $\partial \Gamma$ instead of $\pp \Omega$.
Recall the expression of maximal extension of $\mathfrak{b}$,
\bb
\mathfrak{h}^{\max}(x):=\min_{y\in \pp\mathcal{D}(\Omega)}(
\mathfrak{b}(y)+2\|(x-y)\|_{\infty})
\ee
\label{max_height}
the lattice analog of it is the following,
\bb
\mathfrak{h}_{\#,N}^{\max}(x):=\min_{y\in \pp\mathcal{D}(\Gamma)}(
\mathfrak{b}(y)+2\|(x-y)\|_{\infty}).
\ee
By the Lipschitz condition, $\mathfrak{h}_{\#,N}^{\max}$ approximates $h^{\max}$ up to an error of order $O(N^{-1})$.
Denote approximations obtained this way by subscript $\#,N$. Now, let us prove \hyperref[maxext]{Proposition \ref{maxext}}.

\begin{proof}
The maximal extension of $B_N$ is
\bb
H_N^{\max}(x):=H^{\max}/N=\min_{y \in \pp\mathcal{D}(\Gamma)}(B_N(y)+\beta_N(x,y)).
\ee
It approximates $\hg(x)$ up to $O(N^{-1}$) due to the fact that $|\beta_N(x,y)-2\|(x-y)\|_{\infty}|^{\Omega}\leq \frac{C}{N}$ for $C>0$ independent of $x,y$. Thus, it approximates $\mathfrak{h}^{\max}$.
\end{proof}

\subsection{Density lemma}
Now we are ready to prove the density Lemma, In the proof we use the following tautological way to express a Lipschitz function in terms of its own values,
\bb
\mathfrak{h}(x):=\min_{y \in \mathcal{D}(\Omega)}(\mathfrak{h}(y)+2\|(x-y)\|_{\infty}^{\Omega}).
\label{support123}
\ee
We use the lattice version of the above expression for the lattice $\frac{1}{N}\ZZ^2$ that gives us an approximation of the function.
\bb
\mathfrak{h}_{\#,N}(x):=\min_{y \in \mathcal{D}(\Gamma)_N^{\#}}(\mathfrak{h}(y)+2\|(x-y)\|_{\infty}^{\Omega})
\label{support12}
\ee
It is clear that $\mathfrak{h}_{\#,N}$ approximates $h$ to within an error of order of $O(N^{-1})$ due to the Lipschitz condition.

\begin{theorem}[The Density lemma]
Let $(\Gamma_N, B_{N} )$ be an approximation of $(\Omega, \mathfrak{b})$.
Then, $\exists C>0$ such that for every $\mathfrak{h} \in \mathscr{H}(\Omega,\mathfrak{b})$ there exists a sequence of normalized height functions $H_N$, such that $\norm{\mathfrak{h}-H_{N} }\leq \frac{C}{N}$.

Vice versa, $\exists C^{\prime}>0$ such that for every normalized height function $H_{N}$ on $(\Gamma_N,B_N)$ there exists an asymptotic height function $\mathfrak{h} \in \mathscr{H}(\Omega,\mathfrak{b})$ such that $\norm{\mathfrak{h}-H_{N}} \leq N^{-1}C^{\prime}$.
\label{density}
\end{theorem}

\begin{proof}
The strategy is to use (\ref{support123}) and (\ref{support12}) for the function $\mathfrak{h}$ and its lattice analog for discrete height functions.

Recall that a function $c^{\prime}_y: x\mapsto c^{\prime}+\beta(x,y)$ satisfies the local rule as a function of $x$.
Let us define the partial height function that approximates $h$. For it let us take an infimum over all height functions that are above $\mathfrak{h}$, one way of doing it is to take such a height function at every lattice point, and then take infimum over these functions.
\bb
\hat{H}_{N}(x,\gamma):=\min_{y \in \mathcal{D}(\Gamma)}(\floor*{\mathfrak{h}(y,\gamma_y)}_{p_0}+\beta_N(x,y)),
\ee
where $\floor*{\mathfrak{h}(y,\gamma)}_{p_0}$ means the following: first, we need to take the integer part of $N\times \mathfrak{h}(y,\gamma)$, second we subtract the fractional part of $N\times (h(p_0,\gamma)-\beta(p_0,x))$ modulo $4$, so that the value of height functions under the infimum coincide modulo $4$ at point $p_0$ as they should by \hyperref[mod4]{Lemma \ref{mod4}}, and second we divide it by $N$. These modifications change $\hat{H}_N$ by an error of order $\frac{Const}{N}$, and guarantee that all the lattice operations are well-defined, thus we obtain the height function.

Clearly $\hat{H}_N$ approximates $\mathfrak{h}_{\#,N}$ up to $O(N^{-1})$ (see (\ref{support12})), and thus approximates $\mathfrak{h}$. Note that so far, $\hat{H}_N$ is just a function on vertices of $\Tilde{\Gamma}_N$ with some boundary conditions that satisfy the local rule from \hyperref[defh]{Definition \ref{defh}}. However, we need to check all five properties from \hyperref[hdef]{Definition \ref{hdef}} to obtain the desired height function, so far we have only two of them, the third and the fourth properties. The rest three can be in fact fixed by the right boundary conditions since all these three properties are basically governed by the boundary conditions.

We can “balance” $\hat{H}_N$ between the maximal and the minimal extensions of $B_N$ that may change the height function only by $O(N^{-1})$ due to the fact that $\hat{H}_{N}$ is fit to $\mathfrak{h}$ to within $O(N^{-1})$. After it, we have the desired normalized height function $H_{N}$
\bb
H_{N}:=\max(H_{\min}^N, (\min(H_{\max}^N, \hat H_{N}))).
\ee

The proof of the second part of the statement easier, let us build such a function $\hat{\mathfrak{h}} $ that is simply a linear interpolation of values of $H_N$ at even points (points with even coordinates). We take even points because of the jumps of the height function by $\pm 3$.

\end{proof}

\section{The Concentration Lemma}
\label{sect5}

In this section we prove a concentration inequality for height functions on $\Tilde{\Omega}$. (We suppose that $\Omega$ is a domain, $\mathfrak{b}:\pp\Omega\to\RR$ can be extended to an asymptotic height function on $\Omega$, and $(\Gamma_N, B_N)$ is an approximation of $(\Omega,\mathfrak{b})$)
Also, recall that $\avhn$(p) is the expectation value of $H_N(p)$, then the concentration inequality is the following statement:

\begin{lemma}[The Concentration Lemma]
$\exists \ell(\Omega)>0$ such that  $\forall C>0$ and for sufficiently large $N$
\label{Concentration_lemma}
\bb
\mathbb{P}_N\left(\norma{\hn -\avhn} > C\right)< \exp\left( -\frac{C^2 N}{\ell(\Omega)}\right).
\ee
\label{proba}
\end{lemma}

\begin{proof}
The idea of the proof is to deduce \hyperref[proba]{Lemma \ref{Concentration_lemma}} from the concentration for a simply-connected domain proved in Theorem 21 \cite{CEP}. Let us recall it.

Suppose that $\Gamma$ is a tileable connected simply-connected lattice domain with boundary values of height function $B$, uniform measure $\mathbb{P}^{\Gamma}$ on the set of domino tilings of $\Gamma$.
Take a point $p\in\Gamma$ in the interior of $\Gamma$ such that there is a lattice path from $p$ to $\pp\Gamma$ with $m$ vertices.
The concentration inequality from Theorem 21 in \cite{CEP} is the following statement:

$\forall a>0$,
\bb
\mathbb{P}^{\Gamma}_N \left(| H(p)- \overline{H}(p)|> a \cdot \sqrt[]{m}\right) < 2 \exp(-a^2/32).
\label{concentr}
\ee

Recall the proof of this result.
Let us fix a path $(x_0,\ldots,x_m)$ with $m+1$ vertices from $\pp\Gamma$ to the $p$. Then, the authors considered a filtration of $\sigma$ algebras $\mathcal{F}_k$ generated by outcomes of $H(x_i)$ for $i\leq k$ ($H(x_0)$ is fixed since $x_0\in\pp\Gamma$).
Then, they took the conditional expectation value at point $p$ with respect to $\mathcal{F}_k$, $M_k=\mathbb{E}(H(v)|F_k)$.
By the tower property for conditional expectations, $M_k=\mathbb{E}(H(v)|F_k)$ form a martingale with bounded increments, that is $\mathbb{E}(M_{k+1}|F_k)=M_k$, which after applying Azuma's inequality gives the concentration inequality.

Let us renormalize the concentration inequality (\ref{concentr}) for a large $N$ as follows. Denote $\mathbb{P}_N^{\Gamma}$ the uniform measure on domino tilings of $\Gamma_N$. Then divide by $N$ inequality in the left-hand side of (\ref{concentr}) to get that
\bb
\mathbb{P}^{\Gamma}_N (|\hn(p)-\avhn(p)| > N^{-1}a \cdot \sqrt[]{m}) < 2\exp(-a^2/32).
\ee

Choose $a$ such that $C=N^{-1}a\sqrt{m}$. The length of a path $m$ behaves for large $N$ as $m \approx \ell^{\prime}(\Omega,p) N$. The quantity $\ell^{\prime}(\Omega,p)$ is approximately the length of the shortest path  inside $\Omega$ from point $p\in\Omega$ to $\pp\Omega$. Let us define $\ell^{\prime}(\Omega):=\max_{p} \ell(\Omega,p)$, which is finite due to compactness of $\Omega$.

The resulting concentration inequality so far is the following,

\bb
\mathbb{P}_N^{\Gamma} (|\hn(p)-\avhn(p)| > C) < 2\exp\left(-\frac{N C^2}{32\ell^{\prime}(\Omega)}\right),
\label{bound}
\ee

To obtain the probability $\mathbb{P}^{\Gamma}_N(\norm{\hn-\avhn}\geq  C$), we need to sum over all point $p\in\Gamma$ probabilities that $|\hn(p)-\avhn(p)| \geq C $,

\bb
\mathbb{P}_{N}^{\Gamma}\norm{\hn -\avhn}>C)=\sum_{p \in \Gamma}\mathbb{P}(|\hn(p)-\avhn(p)|\geq C)
\ee
The number of terms in the letter expression is bounded from above by $N^2|\Omega|$, where $|\Omega|$ is the area of $\Omega$, which is due to the fact that the number of vertices of $\Gamma_N$ is approximately $|\Omega|\times N^2 $,

\bb
\mathbb{P}_N^{\Gamma}\left(\norm{\hn -\avhn} > C\right)<2 |\Omega| N^2 \exp\left( -\frac{C^2 N}{32\ell^{\prime}(\Omega)}\right).
\ee

One can further choose an $\ell(\Omega)<32\ell^{\prime}(\Omega))$ to absorb the prefactor $12 |\Omega| N^2$ for sufficiently large $N$, and factor $32$ in front of $\ell^{\prime}$
\bb
\mathbb{P}^{\Gamma}_N\left(\norm{\hn -\avhn} > C\right)<\exp\left( -\frac{C^2 N}{\ell(\Omega)}\right).
\label{prob1}
\ee

Once we have a multiply-connected domain $\Gamma$, the situation slightly changes, however, the original bound still holds. In fact, there are two ways to get the desired inequality.

The first way is a straightforward repetition of the same arguments as in Theorem 21 in \cite{CEP}. One can always find a lattice path from a vertex of $\Tilde{\Gamma}_0$ to its boundary, where boundary values of the height function are fixed, and then apply the same arguments with the Azuma inequality.

Another way is to notice that it is sufficient to show the concentration inequality on $\mathcal{D}(\Omega)$ by \hyperref[iso]{Proposition \ref{iso}}. Since $\Gamma\cap \mathcal{D}(\Omega)$ is a finite simply-connected domain, therefore we can use theorem 21 from \cite{CEP} in the form of \hyperref[proba]{Lemma \ref{Concentration_lemma}}.
\end{proof}

\subsection{Piecewise linear approximations of asymptotic height functions}
In this subsection we recall piecewise linear approximations of Lipschitz functions that we use in the proofs.

Let us take $\ell > 0$ and take a triangular mesh with equilateral triangles of side $\ell$. 
We map an asymptotic height function $h \in \mathscr{H}(\Omega)$ to a piecewise linear approximation, that is linear on every triangle, moreover it is a unique linear function that agrees with $h$ at vertices of the triangle. Let us  denote this approximation of $h$ by $\hat h$. In Lemma 2.2 \cite{CKP} the authors show that in a simply-connected domain, $h$ approximates $\hat h$ on the majority of triangles. In a multiply-connected domain, we can build a piecewise linear approximation of $h$ on $\mathcal{D}(\Omega)$. Moreover, one can use such a triangulation of $\Omega$ that $\pp\mathcal{D}(\Omega)$ consists of sides of the triangles. The resulting approximation $\hat h$ has the same increments between connected boundary components $\nu_i$ and $\nu_{i+g}$ as $h$. The latter fact allows us to extend $\hat h$ to $\Tilde{\Omega}$ with the same monodromy data as $h$ and with the desired approximation property. Thus, we have the following.
\begin{claim} \label{analitic1} 
Let $h \in \mathscr{H}(\Omega)$ be asymptotic height function and let $\varepsilon>0$. Then for sufficiently small $\ell>0$, on at least $1-\varepsilon$ fraction of the triangles in the $\ell$-mesh that intersect $\Omega$, we have the following two properties: first, piecewise linear approximation $h_{\ell}$ is fit to within $\varepsilon \ell$ to $h$. Second, on at least $1-\varepsilon$ fraction(in measure) of points $x$, $\nabla h(x)$ exists and is within $\varepsilon$ to $\nabla h_{\ell}$.
\end{claim}

\subsection{The cutting rule}
Suppose that we have a domain with a boundary condition $(\GT,\mathfrak{b})$ and a subset $\rho$ on the dual lattice from the boundary of $\GT$ to itself (thus, $\Gamma / \rho$ consists of several components, let us denote them $\Gamma^i$). We want to calculate the partition function of $\GT$ and one way to do it is to calculate partition functions $Z(H(\rho))$ with the given height function $H(\rho)$ along $\rho$. Then to sum up $Z(H(\rho))$ over all $H(\rho)$. The result is the original partition function because we just permute terms in a finite sum.

Then, we can interpret each $Z(H(\rho))$ as the product of the partition functions saying that $\rho$ cuts $\GT$.

\bb
Z\GB=\sum_{B_\rho}\prod_{i}Z(\GT^i,B^i_\rho)
\label{Cutting_boundary}
\ee
where $B^i_\rho$ is the boundary height function on $\Gamma^{i}$ that coincide with the original boundary height function $B$ and $B_\rho$ where it is possible.

\subsection{Surface tension}
\label{Surf}
Recall from \cite{CKP} that the asymptotic growth rate $\sigma(s,t)$ of the number of domino tilings of a rectangle $N\times N$ with periodic boundary conditions with the slope $(s,t)$ is defined by the following formula as $N\to\infty$, 
\bb
Z(R_N(s,t))\approx \exp\left(N^2\sigma(s,t)+O(N)\right).
\ee
The precise expression of $\sigma(s,t)$ is the following.
Recall the Lobachevsky function $L(z)=-\int_{0}^{z}{\log |2\sin{t}|dt}$ and quantities $p_a=p_a(s.t),p_b=p_b(s.t),p_c=p_c(s.t),p_d=p_d(s.t)$ that are determined by \hyperref[sigma]{expression (\ref{sigma})}. Then, the surface tension is the following function,
\bb
\sigma(s,t)= 1/\pi\left( L(\pi p_a)+L(\pi p_b)+L(\pi p_c)+L(\pi p_d) \right).
\label{Surfacedef}
\ee

The probabilities of four types of dominoes $p_a,p_b,p_c$ and $p_d$ that are determined by the following system in the limit as $N\to\infty$ \cite{CKP},
\bb
\begin{aligned}
&\quad 2(p_a - p_b) = t,\\
&\quad 2(p_d - p_c) = s,\\
&\quad
p_a + p_b + p_c + p_d = 1\\
&\quad
\sin(\pi p_a) \sin(\pi p_b) = \sin(\pi p_c) \sin(\pi p_d).
\label{sigma}
\end{aligned}
\ee

\section{The variational principle}
\label{sect6}

In this section we formulate and prove \hyperref[lim]{Theorem \ref{lim}} and its corollary,\hyperref[lim2]{Theorem \ref{lim2}}.

\subsection{Statement of theorems}
Suppose $\Omega$ is a domain with an boundary condition $\mathfrak{b}$ and let $(\Gamma_N,B_N)$ be an approximation of $(\Omega, \mathfrak{b})$, further $r$ is a continuous height change.

We also need to assume that the boundary condition $\mathfrak{b}$ is non-degenerate, that is it admits an extension to an asymptotic height function $\mathfrak{h}$ whose gradient is in the interior of the newton polygon $\mathcal{N}$ on the set of positive measure.
We need to assume this to guarantee the uniqueness of the limit shape $\mathfrak{h}^{\star}$, which is not the case for pathological boundary conditions, which could have only linear extensions of slope $\pp\mathcal{N}$, i.e., extensions with only frozen regions, where we can not use concavity of $\sigma$.

Also recall that $H_N:=\frac{1}{N}H$ is a normalized height function together with the normalized height change $\frac{1}{N}R_N$. 
Finally, let $\mathfrak{h}^{\star}$ be the unique maximizer of $\mathcal{F}$ over $\mathscr{H}(\OO,\mathfrak{b})$. Here we

and let $r^{\star}$ be the continuous height change of $\mathfrak{h}^{\star}$ Then,

\begin{theorem}
 In the limit as $N\to\infty$ and then $\delta\to 0$
\bb
  |\Omega|^{-1}N^{-2} \log Z \left( \Gamma_N,B_N \right) = \mathcal{F}(\mathfrak{h}^{\star})= \iint_{\Omega}{\sigma (\nabla \mathfrak{h}^{\star}) dx dy}+o_\delta(1).
\ee
Moreover, there exists $\ell>0$ such that as $N\to\infty$ we have
\bb
\mathbb{P}_N\left( \max_{x_N\in\Gamma}|H_N(x_N)-\mathfrak{h}^{\star}(x_N)|>\delta \right)\leq  \exp\left(-\frac{N\delta^2}{\ell}\right).
\ee
and $R_N$ converge to $r^{\star}$ with respect to $\mathbb{P}_N$ as $N\to\infty$.
\label{The_limit}
\end{theorem}
This theorem can be formulated for a fixed height change $r\in\RR^g$, assume also that $\frac{1}{N} R_N \to r$ as $N\to\infty$. Then we have the following Corollary,
\begin{theorem}
\label{lim2}
 There exists $\mathfrak{h}^{\star}_{r}\in\mathscr{H}(\OO,\mathfrak{b},r)$ such that in the limit $N\to\infty$ and then $\delta\to 0$
\bb
    |\Omega|^{-1}N^{-2} \log Z \left( \Gamma_N,B_N, R_N \right) =  \iint_{\Omega}{\sigma (\nabla \mathfrak{h}^{\star}_r) dx dy}+o_\delta(1).
\ee
Furthermore, there exists $\ell>0$ such that as $N\to\infty$ the following holds,
\bb
\mathbb{P}_N^R\left( \max_{x_N\in\Gamma}|H_N(x_N)-\mathfrak{h}^{\star}_r(x_N)|>\delta \right)\leq \exp\left(-\frac{N\delta^2}{\ell}\right).
\ee

\end{theorem}

In the remaining section we prove \hyperref[The_limit]{Theorem \ref{The_limit}} in three steps, and deduce from it \hyperref[lim2]{Theorem \ref{lim2}}. The steps are

\subsection{Convergence of height functions to the limit shape}

Here, we give a proof of a law of large numbers for height function. More precisely, we show that a normalized height function converges in both regimes to its expected value that is approximately the unique solution to the variational problem. We write it for an arbitrary height change, and later discuss the modifications for a fixed height change.

Consider a sequence $\{ \overline{H}_N \}$ of expectation values of normalized height functions on $(\Gamma_N,B_N)$. We know that by \hyperref[density]{Theorem~\ref{density}} there exists a sequence of asymptotic height functions $\{\mathfrak{h}_N\}$, such that $\norm{\mathfrak{h}_N-\avhn}\leq \frac{C}{N}$. By \hyperref[comp]{Proposition \ref{comp}}  $\{\mathfrak{h}_N\}$ has a convergent subsequence, denote its limit by $\mathfrak{h}$. Without loss of generality, we suppose that the convergent subsequence is $\{\mathfrak{h}_N\}$ itself. We will see later that $\mathfrak{h}=\mathfrak{h}^{\star}$

\begin{lemma}
 One can find such an $\ell^{\prime}$ that for $\forall\delta>0$ and for sufficiently large $N$
\bb
\mathbb{P}(\norm{H_N-\mathfrak{h}}>\delta)\le \exp(- N \delta^2/\ell^{\prime}).
\ee

\end{lemma}
\begin{proof}
We deduce it from a combination of \hyperref[concentration]{Lemma \ref{Concentration_lemma}} applied to $H_N$ and convergence of $h_N$ to $\mathfrak{h}$. The goal is to show that the inequality $\norm{H_N-\mathfrak{h}}\leq\delta$ holds for sufficiently large $N$ with probability exponentially close to $1$.

By the triangle inequality we have

`\bb
\norm{H_N-\mathfrak{h}}\leq \norm{\bar{H}_N-\mathfrak{h}_N}+\norm{\mathfrak{h}_N-\mathfrak{h}}+\norm{H_N-\bar{H}_N}.
\label{triangle}
\ee

The first term in the right-hand side of (\ref{triangle}) is smaller than $\delta/3$ for sufficiently large $N$ by definition of $\mathfrak{h}_N$, that is by \hyperref[density]{Lemma \ref{density}} we have $C>0$ such that $\norm{\mathfrak{h}_N-\bar{H}_N}<\frac{C}{N}$, which is smaller than $\delta/3$ for sufficiently large $N$.
The second terms of (\ref{triangle}) is smaller than $\delta/3$ for sufficiently large $N$ due to convergence of $\mathfrak{h}_N$ to $\mathfrak{h}$, which holds by definition of $\mathfrak{h}$.

The third term (\ref{triangle}) is smaller than $\delta/3$ with the probability $1-\exp(-\frac{\delta^2 N}{\ell})$, where we defined $\ell=3\ell^{\prime}$ for $\ell^{\prime}$ from \hyperref[concentration]{Lemma \ref{Concentration_lemma}}.
Therefore, we have $\norm{H_N-\mathfrak{h}}>\delta$ with probability bounded by $ \exp(- N \delta^2/\ell^{\prime})$.

\end{proof}

So far, we know that up to extracting a subsequence, normalized height functions converge with respect to the uniform norm in probability to $\mathfrak{h}$ and the contribution of height functions that are far away from $\mathfrak{h}$ are exponentially suppressed. Thus, $\mathfrak{h}$ is the limit shape. Since height change $R_N$ is a continuous function of $H_N$, it converges to the height change of the limit shape, which is $r$. Later, we prove that $\mathfrak{h}=\mathfrak{h}^{\star}$.

\subsection{Convergence of partition function}
Let us show that one can find an asymptotic expression of the partition function, which is a straightforward corollary of the convergence of height functions to the limit shape. The following proof holds for a fixed height change $r$ after replacing $\mathfrak{h}$ by $\mathfrak{h}_r$ and $Z(\Gamma_N,B_N)$ by $Z(\Gamma_N,B_N,R_N)$.

Define $U^N_{\delta}(\mathfrak{h}^{\star})$ to be the set of height functions on $\Gamma_N$ that are fit to within $\delta$ to $\mathfrak{h}^{\star}$, $\norm{H_N-\mathfrak{h}^{\star}}\leq\delta$.

Then, the following holds due to the concentration inequality above,

\bb
\mathbb{P}_N\norm{H_N-\mathfrak{h}^{\star}}\leq \delta)=1-\mathbb{P}_N(\norm{H_N - \mathfrak{h}^{\star}}>\delta).
\ee
Or more precisely,
\bb
\frac{Z(\Gamma_N,B_N|\mathfrak{h}^{\star},\delta)}{Z(\Gamma,B_N)}=1+O(
\exp(-\delta^2 N/\ell)).
\ee

Now, let us take the logarithm of both sides and normalize them by $N^{-2}$. Also introduce the notation $S(N,\delta):=1+O(
\exp(-\delta^2 N/\ell))$ for the simplicity.
\bb
N^{-2}\log Z(\Gamma_N,B_N)=N^{-2}\log(Z(\Gamma_N,B_N|\mathfrak{h}^{\star},\delta))+N^{-2}\log(S(N,\delta)).
\ee

Now we can take limit as $N\to\infty$ to make $\log S(N,\delta)$ converge to zero. Finally, we obtain the desired expression by \hyperref[Surface]{Theorem \ref{Surface}},
\bb
N^{-2}\log Z(\Gamma_N,B_N)\xrightarrow[{N\to\infty}]{}\mathcal{F}(\mathfrak{h}^{\star}).
\ee

\subsection{The Surface tension functional and the limit shape}

We still need to show that $h$ maximizes $\mathcal{F}$, that is $h=\mathfrak{h}^{\star}$. Again, it follows easily from the concentration of height functions around $\mathfrak{h}$.

Assume that $\mathfrak{h}\neq \mathfrak{h}^{\star}$. We need to show that $\mathcal{F}(\mathfrak{h}) \geq \mathcal{F}(\mathfrak{h}^{\star})$.
Suppose the opposite, that is $\mathcal{F}(\mathfrak{h}) < \mathcal{F}(\mathfrak{h}^{\star})$. By \hyperref[functional]{Theorem \ref{Functional}}, $\mathcal{F}(\mathfrak{h}^{\star})$ (resp.$\mathcal{F}(\mathfrak{h})$) is the limit of the normalized number of domino tilings whose normalized height functions are fit  within $\delta$ to $\mathfrak{h}^{\star}$ (resp.$\mathfrak{h}$) for $\delta\to 0$. Then, we can use that normalized height functions $H_N$ concentrate around $\mathfrak{h}$, which can be separated from $\mathfrak{h}^{\star}$ by the choice of a smaller $\delta>0$. The contradiction follows from the fact that the overwhelming majority of domino tilings are $\delta$-close to $\mathfrak{h}$, but not to $\mathfrak{h}^{\star}$ and thus, we are done.

The proof for the fixed asymptotic height change $r$ works the same way after replacing of $\mathfrak{h}$ by $\mathfrak{h}_r$.

\section{Existence of the maximizer}
Here we show the existence of solution of the variational problems, for an arbitrary height change, and for a given asymptotic height change.

Recall that the gradient of asymptotic height functions
Before formulating the theorems, we need an extra assumption on
the boundary conditions. Recall that $\sigma: \mathcal{N}\to\RR$ is concave only in the interior of $\mathcal{N}$, thus in order to guarantee the uniqueness, we must have at least one point where $\nabla \mathfrak{h}\in\\mathcal{N}^{\circ}$. So call a boundary condition non-degenerate if it admits an extension to such an asymptotic height function $\mathfrak{h}$ that $\exists$ such a proper subset $\Omega^{\prime}\subseteq \Omega$ that $\forall p\Omega^{\prime}$ $\nabla(p) \mathfrak{h}\in\mathcal{N}^{\circ}$. This way we have a subregion of $\Omega$ where we can use the concavity of $\sigma$ following proposition 2.4 from \cite{CKP}.
It is easy to construct counterexamples in the situation where this condition does not hold, and the boundary condition admits two different totally frozen extensions, see discussion of figure 39 from \cite{NSW}.
\label{newton}

\label{sect7}
\begin{proposition}
There exists a unique maximizer $\mathfrak{h}^{\star}$ of $\mathcal{F}$ over the space $\mathscr{H}(\Omega,\mathfrak{m},\mathfrak{b})$. 

There exists a unique maximizer $\mathfrak{h}^{\star}_r$ of $\mathcal{F}$ over the space $\mathscr{H}(\Omega,\mathfrak{m},\mathfrak{b}_r,r)$.
\label{exist}
\end{proposition}

\begin{proof}
Consider a fundamental domain $\mathcal{D}(\Omega)$ for action of $\pi _1(\Omega)$ on the universal covering space of $\Omega$. Recall that by \hyperref[unique]{Proposition \ref{iso}}, it is sufficient to find a maximizer of $\mathcal{F}$ on $\mathcal{D}(\Omega)$. 

The boundary $\pp \mathcal{D}(\Omega)$ consists of two parts, $\pp \mathcal{D}(\Omega)=\pp \Omega\bigsqcup \pp \mathcal{D}^{1}(\Omega)$, the first one is supplemented with boundary condition $\mathfrak{b}$ and the other one is the union of $2g$ curves $\{\upsilon_i \}_{i=1}^{2g}$ matched in pairs $\upsilon_i, \upsilon^{\prime}$ with free boundary conditions subject to periodicity across the pairs so that the values on $\upsilon_i$ coincide with the values on $\upsilon^{\prime}_i$(recall that $\upsilon_i$ and $\upsilon^{\prime}_i$ were obtained by the cut along the curve $\gamma_i$)

The space $\mathscr{H}(\Omega,\mathfrak{m},\mathfrak{b})$(resp.~$\mathscr{H}(\Omega,\mathfrak{m},\mathfrak{b}_r,r)$) is compact by \hyperref[comp]{ Theorem \ref{comp}}. Then, $\mathcal{F}$ is upper semi-continuous on spaces $\mathscr{H}(\Omega,\mathfrak{m},\mathfrak{b}_r,r)$ and $\mathscr{H}(\Omega,\mathfrak{m},\mathfrak{b})$. This follows from Lemma 2.1 from \cite{CKP}, which uses only the convexity of $\sigma$ and the Lipschitz condition, and thus, trivially extends to partially free periodic boundary conditions.

Therefore, there exists the maximizer $\mathfrak{h}^{\star}$ of $\mathcal{F}$ on $\mathscr{H}(\Omega),\mathfrak{m},\mathfrak{b})$ (resp. $\mathfrak{h}^{\star}_r$ on $\mathscr{H}(\Omega,\mathfrak{m},\mathfrak{b}_r,r)$). By the same convexity argument as in Lemma 2.4 from \cite{CKP}, this maximizer is unique.

A priori $\mathfrak{h}^{\star}$ and $\mathfrak{h}^{\star}_r$ depend on a particular choice of the fundamental domain $\mathcal{D}(\Omega)$. However, since gradients of functions from $\mathscr{H}(\Tilde{\Omega},\mathfrak{b})$ are well-defined objects on $\Omega$ by \hyperref[well]{Lemma \ref{well}} and $\sigma$ is strictly convex everywhere in the interior of the domain of definition \cite{CKP,KOS}, we can use Proposition 4.5 from \cite{DS} to show uniqueness. See also the discussion in Section 5.6 in \cite{ZAPD}.
\end{proof}

\section{The Surface Tension Theorem.}
\label{sect8}
In this section we formulate \hyperref[Surface]{Theorem \ref{Surface}} and give a proof of it.

\begin{theorem}
\label{Surface}
Let $\Omega$ be a domain in $\mathbb{R}^2$ and $\mathfrak{h} \in \mathscr{H}(\Omega,\mathfrak{b})$. Suppose that $(\Gamma_N,B_N)$ is an approximation of $(\Omega,\mathfrak{b})$.

 Then, for $\forall \delta>0$ sufficiently small,

\begin{equation} 
\lim_{N\to\infty}\sup\left|\frac{1}{|\Omega|}N^{-2} \log Z(\Gamma_N,B_N\mid \mathfrak{h},\delta)-\int_{\Omega}{\sigma(\pp_x \mathfrak{h},\pp_y \mathfrak{h}) dx dy}\right|=o_\delta(1),
\end{equation}
\label{Functional}
\end{theorem}
\begin{proof}
Fix a fundamental domain $\mathcal{D}(\Omega)$ with branch-cuts made along curves $\{\gamma_i\}_{i=1}^g$.

For the proof we need a triangular mesh a side length $\ell$ and piecewise linear approximations of Lipschitz functions from \hyperref[analitic1]{Claim \ref{analitic1}}. 
Consider a triangular mesh of length size $\ell$ that triangulates $\mathcal{D}(\Omega)$ into triangles $T^j$ of the standard area $\mathcal{A}(T^j)$. Let also $\mathfrak{h}_\ell$ be the piecewise linear approximation of $\mathfrak{h}$ that is linear on each triangle and coincides with the values of $h$ at the vertices of the triangles.

Then, we choose small $\varepsilon$ and take $\ell$ such that $\ell\varepsilon<\delta$, and on at least at $1-\epsilon$ fraction of points of a triangle in measure we have two properties: first, 
$|\mathfrak{h}_\ell-\mathfrak{h}|\leq \ell\varepsilon$, and second, the gradient of $\mathfrak{h}$ exists and fit within $\varepsilon$ to the gradient of $\mathfrak{h}_\ell$ in $\ell_2$ norm, which is possible by Lemma 2.2 \cite{CKP}. We need these properties in order to write $\mathcal{F}(\mathfrak{h}_\ell)=\mathcal{F}(\mathfrak{h})+o_\delta (1)$ as $\varepsilon\to 0.$

Let us use the $\hyperref[Cutting_boundary]{(\ref{Cutting_boundary})}$ for a subset $\rho$ obtained from the intersection of the triangular mesh with the domain. Note that the subset $\rho$ cuts $\Gamma_N$ into triangles with boundary conditions along $\rho$. Denote the triangles by $\{T^{j}\}$ and their boundary height functions by $\{B_{\rho}^{j}\}$.
\bb
Z(\Gamma_N,B_N|\mathfrak{h}, \delta)=\sum_{\mathfrak{b}_\rho}{Z(\Gamma_N,B_{\rho}|\mathfrak{h}, \delta)}
=\sum_{\mathfrak{b}_\rho} \prod_{j}{Z(T^{j},B_{\rho}^{j}| \mathfrak{h}, \delta)},
\label{cutting_sum}
\ee

There are two types of triangles $\{T^j\}$. The included triangles (the first type) that do not intersect the boundary of the fundamental domain and where $\mathfrak{h}_\ell$ is fit to within $\ell\varepsilon$ to $\mathfrak{h}$. The excluded triangles (the second type) intersect the boundary of the fundamental domain or where $\mathfrak{h}_\ell$ does not approximate $\mathfrak{h}$.

We make an upper and a lower bound for the normalized partition function \newline $N^{-2}{\log Z(\Gamma_N,B_N|\mathfrak{h},\delta)}$. In both cases, we can estimate the two types of triangles separately. For the included triangles we use Corollary 4.2 from \cite{CKP}, and for the excluded we make a rough estimate. Then, after taking limit as $N\to\infty$, the normalized estimates differ from each other by $o_\delta(1)$.

\subsection{The lower bound.}

In the lower bound, it is sufficient to include some height functions that are $\delta$-close to $\mathfrak{h}$. 
To do this, we can take only one term from \hyperref[cutting_sum]{(\ref{cutting_sum})} corresponding to one boundary height function $B_\rho$ (for instance, we can take $B_\rho$ obtained from the restriction of $\hn^{\prime}$) (\hyperref[density]{Theorem~\ref{density}} gives us a sequence of normalized height functions $\{\hn^{\prime} \}$ that converges to $\mathfrak{h}$ as $N\to \infty$).

Let us estimate the triangles of the first type by the product that includes only triangles of this type.

\bb
Z(\Gamma_N,B_N|\mathfrak{h},\delta)\geq \prod_{T_j \text{of 1st type }}{Z(T^{j},B_{\rho}^{j}| \mathfrak{h}_\ell, \delta)}
\ee

The bound for the included triangles obtained by using Corollary 4.2 \cite{CKP} to count $\delta$-close height functions to make sure that we include only height functions $\delta$-close to $h$. Thus, for triangles of the first type, we have the following,

\bb
N^{-2}\log \prod_{T_j \text{of 1st type }} Z(T^{j},B_{\rho}^{j}|\mathfrak{h}_\ell,\delta)=\sum_{T_j \text{of 1st type }}\sigma(s^j,t^j) \times \cA(T^{j})+o(N^{-1})+O(\varepsilon^{1/2}\log{\varepsilon}),
\ee

where $(s^j,t^j)$ is a slope of $\mathfrak{h}_{\ell}$ on the triangle $T^j$ and $\cA(T^{j})$ is the area of the triangle $T^j$. 

Finally, for sufficiently large, $N$ the lower bound is the following, 

\bb
N^{-2}\log Z(\Gamma_N,B_N|\mathfrak{h},\delta)\geq \sum_{j} \sigma(s^j,t^j) \cA(T^{j})+o(N^{-1})+O(\varepsilon^{1/2}\log{\varepsilon}),
\ee
where we fixed a height function on the excluded triangles to be $\hn^{\prime}$.

\subsection{The upper bound.}

We can use almost the same strategy to make an upper bound. First, we have to include all height function $\delta$-close to $\mathfrak{h}$. Let us estimate the included triangles the same way as for the lower bound to count height functions $\delta$-close to $\mathfrak{h}$. For the excluded triangles we make a rough estimate since the area of those triangles is proportional to $\varepsilon$, the number of configurations is bounded from above by $\exp(\varepsilon)$ times the number of terms in the cutting rule, which is $2^{O(N)}$ since it is the number of configurations of a line of size $N$.

\bb
Z(\Gamma,B_N|\mathfrak{h},\delta)\leq \prod_{T_j \text{of 1st type }}{Z(T^{j},B_{\rho}^{j}|\mathfrak{h}_\ell,\delta))}2^{O(N)}+N^2\varepsilon
\ee
And after taking limit as $N\to\infty$, the normalized upper bound is the following, where terms $O(N^{-1})$ and $O(\varepsilon)$ come from the bounds of triangles of the second type.

\bb
N^{-2}\log Z(\Gamma,B_N|\mathfrak{h},\delta)\leq\sum_{j} \sigma(s^j,t^j) \cA(T^{j})+o(N^{-1})+O(\varepsilon^{1/2}\log{\varepsilon})+O(\varepsilon).
\ee

Taking into account that $\int_{\Omega}{\sigma(\nabla h_{\ell}) dx dy}=\sum_{j}{\sigma(s_j,t_j)}\mathcal{A}(t^j)$, one can see that both bounds after dividing by the area of $\Omega$ are equal to $\mathcal{F}(h_{\ell})+o(N^{-1})+O(\varepsilon^{1/2}\log{\varepsilon})$ that differs from $\mathcal{F}(\mathfrak{h})$ by $o_\delta(1)$ by Lemma 2.3 from \cite{CKP}.
\end{proof}

\newpage

\section{Final remarks}
\label{sect9}
\subsection{}

\label{results}

As mentioned in \cite{VG}, there are two approaches in studying random domino tilings of a multiply-connected domain $\Gamma$ that are equivalent for a simply-connected domain.

The first approach is looking at domino tilings of $\Gamma$ with the uniform measure $\mathbb{P}$ defined on them. In this framework, one might be interested in fluctuations either of a height function or a height change. For instance, in \cite{BGG} the authors showed Gaussian fluctuation of normalized height change $\frac{1}{N}R_N$ using the method of log-gaze.

The second option is fixing a boundary height function $B^R$ with the height change $R$ and looking at uniformly random height functions that extend $B^R$ to $\Gamma$. Denote $\mathbb{P}_N^R$ the uniform measure on such extensions, which is just $\mathbb{P}_N$ conditioned to have the fixed height change $R$. This approach suits a random surface point of view on tilings, where we look at a plot of a height function as a random stepped surface. Computer simulation of domino tilings with different height change in \hyperref[simul]{\ref{simul}} shows that this parameter is extremely important.
Results in this direction include the first description of a non-simply-connected domain in \cite{BG}. The authors proved a law of large numbers and a central limit theorem for domino tilings of so-called holey Aztec diamond. Up to our knowledge \cite{BG} is the only work that deals with multivalued height functions, yet the authors do not find $\mathfrak{h}^{\star}$ explicitly or characterize it besides the law of large numbers. Other works focused on a problem of random lozenge tilings of multiply-connected domains with monodromy-free height functions. In \cite{KO} the analysis using the complex Burgers equation was done with an example of a frozen(Arctic) curve in a non-simply-connected region. In recent years there also appeared combinatorial works with enumerating results by M. Ciucu et al. for example see \cite{CL}. Further results are obtained using the tangent method by P.Di Francesco et al. in \cite{DFG}, where the authors have found the frozen curve for quarter-turn symmetric domino tilings of a holey Aztec square, which is a multiply-connected domain with a hole of a finite size. 

Also, the idea of defining a height function on $\Tilde{\Omega}$ in a different notation is already known, for instance \cite{BLR}.

\subsection{}
The case of a hole of a finite size, as in \cite{DFG}, can be, probably, analyzed in our framework as follows. Since the hole converges with respect to Hausdorff distance to a point $(x_0,y_0)$ as $N\to\infty$, we left with one parameter that
encodes the height of $(x_0,y_0)$. So, we need to modify the space of function by fixing the value of height functions at $(x_0,y_0)$. As the result, one would expect conic singularities of the limit shape as it is approaching the point $(x_0,y_0)$, as in a similar example noted in \cite{KO}.

\subsection{}
Let us recall that a flip of a domino tiling is a replacement of two adjacent vertical dominoes by two adjacent horizontal dominoes. The property of domino tilings of a simply-connected region $\Gamma$ is that any two domino tilings are related by a sequence of flips \cite{T,STCR}. In other words, the set of domino tilings forms one orbit under the action of flips. This property is in the core of computer simulations of random domino tilings \cite{PW}. Thus, the simulation algorithm of uniformly random domino tilings should include an extra move that change $R$. This move is a cyclic rotation of dominoes along the non-trivial loop. See details in \cite{DFG}.

\subsection{}
Let us define our main example, the modified Aztec diamond $\mathcal{AD}_{N}$. Also, we explain certain features of $\mathcal{AD}_{N}$ and possible ways to analyzing it.
\begin{figure}
    \includegraphics[width=0.3\linewidth]{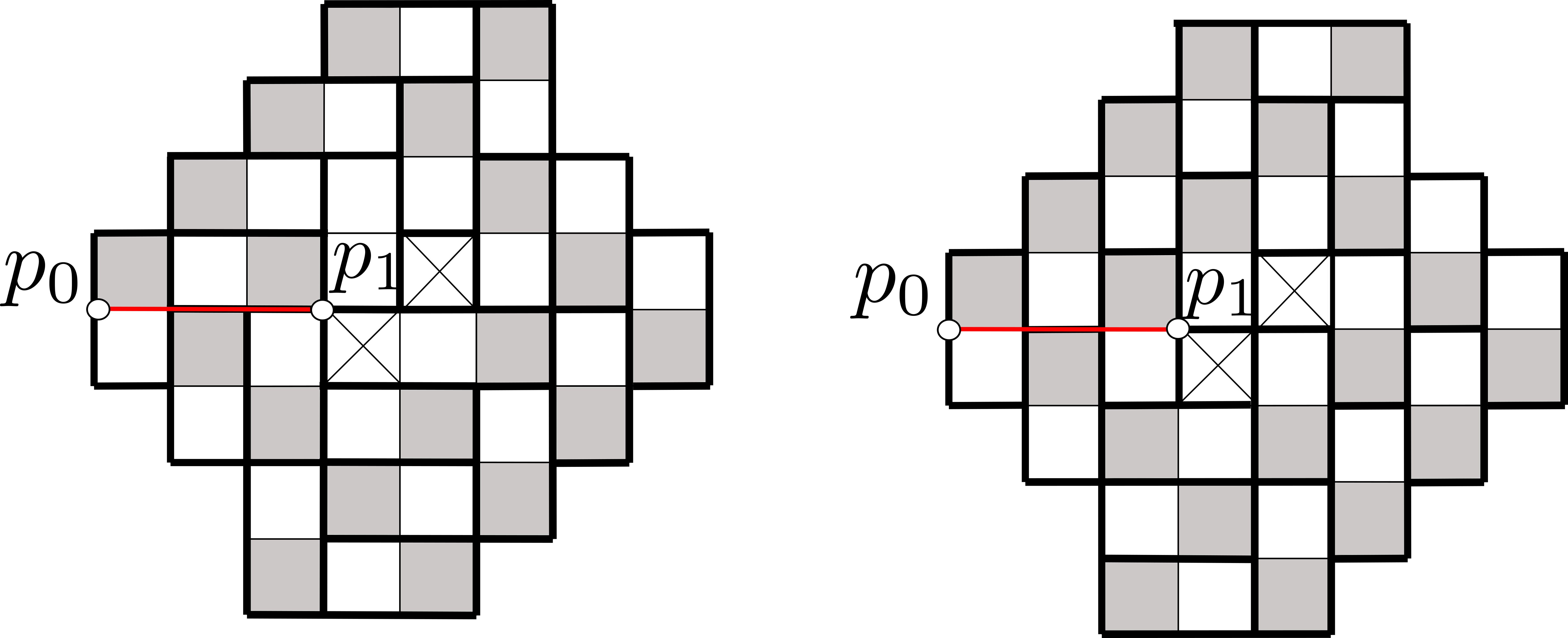}
    \caption{Two domino tilings of $\mathcal{AD}_{1}$ with height change $R_1=3$ on the right figure and $R_1=7$ on the left. The crossed squares are missing from the domain.}
    \label{island}  
\end{figure}

\label{modi}
Recall that the Aztec diamond of order $N$ is the union of unit squares on $\ZZ^2$ whose centers $(x,y)$ satisfy $|x|+|y|\leq N$. Let $N=4k, k\in\mathbb{N}$ and introduce the Aztec diamond with a modified constraint $AD^{\circ}_N$, $N/4 \leq|x|+|y|\leq N$. The boundary of $AD^{\circ}_N$ consists of two connected components, the external boundary and the internal one. For our main example we make four defects to the latter boundaries, that is consider $AD^{\circ}_N$ and add $N/4$ squares in the following four locations, right upper and left bottom external boundaries(resp. left upper and right bottom internal boundaries).  See an example of $\mathcal{AD}_{1}$ on \hyperref[island]{Figure~ (\ref{island})}. A height function on this domain has monodromy $M=8$. It is not hard to check using a checkerboard coloring that $\mathcal{AD}_{N}$ is tillable for arbitrary $N=4k, k\in\mathbb{N}$.

One interesting property of $\mathcal{AD}_{N}$ is an emergence of two paths on the top and on the bottom of it that can be clearly seen on \hyperref[fancy]{Figure~ (\ref{fancy})}. These paths exist in all the domino tilings of $\mathcal{AD}_{N}$, which can be seen from the parametrization of domino tilings by non-intersections paths via bijection with non-intersecting line ensemble as in \hyperref[NILP]{Figure~ \ref{NILP}}.

Recall that a frozen region is the set of points of $\Omega$ where fluctuation of $H_N$ disappears as $N\to\infty$, the boundary of the frozen region is called a frozen(arctic) curve. The paths mentioned above approximate the tangent lines to the arctic curve. This property is in the core of the heuristics of the tangent method \cite{CS}, which reconstructs the arctic curve from its tangent lines. Recently, this method was proved for a particular case of the six-vertex model, the ice-model on a three-bundle domain \cite{AG2}. We think that this technique can be used to determine the Arctic curve in our situation. We leave it for future work.
\begin{figure}
    \centering
    \includegraphics[width=0.3\linewidth]{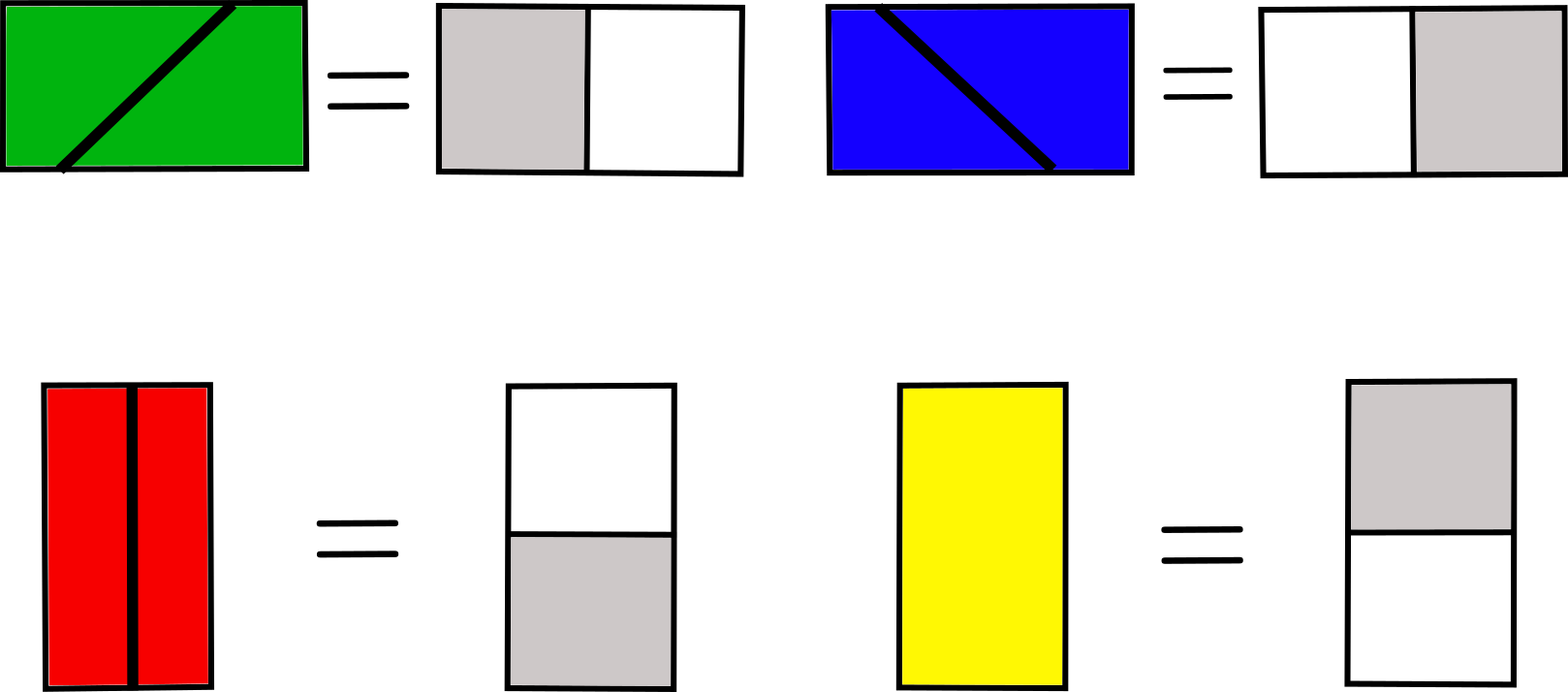}
    \caption{Bijection between domino tilings and non-intersecting line paths}
    \label{NILP}
\end{figure}
More interestingly, one can modify the definition of $\mathcal{AD}_{N}$ by changing the size of the defects and obtain a one-parameter family of domains and frozen curves. See \hyperref[simul]{Section~\ref{simul}} for simulations of $\mathcal{AD}_{N,M}$ with different defects.

\addcontentsline{toc}{section}{References}
\bibliographystyle{alpha}
\bibliography{main}

\begin{thebibliography}{DDFG20}

\bibitem[ADPZ20]{ZAPD}
Kari Astala, Erik Duse, István Prause, and Xiao Zhong.
\newblock Dimer models and conformal structures, 04 2020.

\bibitem[Agg19]{AG2}
Amol Aggarwal.
\newblock Universality for lozenge tiling local statistics, 07 2019.

\bibitem[Agg20]{AG}
Amol Aggarwal.
\newblock Arctic boundaries of the ice model on three-bundle domains.
\newblock {\em Inventiones mathematicae}, 220, 05 2020.

\bibitem[BG19]{BG}
Alexey Bufetov and Vadim Gorin.
\newblock Fourier transform on high-dimensional unitary groups with
  applications to random tilings.
\newblock {\em Duke Math. J.}, 168(13):2559--2649, 2019.

\bibitem[BGG17]{BGG}
Alexei Borodin, Vadim Gorin, and Alice Guionnet.
\newblock Gaussian asymptotics of discrete {$\beta$}-ensembles.
\newblock {\em Publ. Math. Inst. Hautes \'{E}tudes Sci.}, 125:1--78, 2017.

\bibitem[BLR19]{BLR}
Nathanaël Berestycki, Benoît Laslier, and Gourab Ray.
\newblock The dimer model on riemann surfaces, i, 08 2019.

\bibitem[CEP96]{CEP}
Henry Cohn, Noam Elkies, and James Propp.
\newblock Local statistics for random domino tilings of the {A}ztec diamond.
\newblock {\em Duke Math. J.}, 85(1):117--166, 1996.

\bibitem[CKP00]{CKP}
Henry Cohn, Richard Kenyon, and James Propp.
\newblock A variational principle for domino tilings.
\newblock {\em Journal of the American Mathematical Society}, 14, 08 2000.

\bibitem[CL19]{CL}
Mihai Ciucu and Tri Lai.
\newblock Lozenge tilings of doubly-intruded hexagons.
\newblock {\em J. Combin. Theory Ser. A}, 167:294--339, 2019.

\bibitem[CN23]{NSW}
and Catherine~Wolfram Chandgotia~Nishant, Scott~Sheffield.
\newblock Large deviations for the 3d dimer model.
\newblock {\em arXiv preprint arXiv:2304.08468}, 2023.

\bibitem[CS16]{CS}
Filippo Colomo and Andrea Sportiello.
\newblock Arctic curves of the six-vertex model on generic domains: The tangent
  method.
\newblock {\em Journal of Statistical Physics}, 164, 09 2016.

\bibitem[DC04]{DC}
S.~V. Duzhin and B.~D. Chebotarevsky.
\newblock {\em Transformation groups for beginners}, volume~25 of {\em Student
  Mathematical Library}.
\newblock American Mathematical Society, Providence, RI, 2004.
\newblock Translated and revised from the 1988 Russian original by Duzhin.

\bibitem[DDFG20]{DFG}
Bryan Debin, Philippe Di~Francesco, and Emmanuel Guitter.
\newblock Arctic curves of the twenty-vertex model with domain wall boundaries.
\newblock {\em J. Stat. Phys.}, 179(1):33--89, 2020.

\bibitem[DSS08]{DS}
Daniela De~Silva and Ovidiu Savin.
\newblock Minimizers of convex functionals arising in random surfaces.
\newblock {\em Duke Math. J.}, 151, 10 2008.

\bibitem[Fou96]{F}
J.~C. Fournier.
\newblock Pavage des figures planes sans trous par des dominos: fondement
  graphique de l'algorithme de {T}hurston, parall\'{e}lisation, unicit\'{e} et
  d\'{e}composition.
\newblock volume 159, pages 105--128. 1996.
\newblock Selected papers from the ``GASCOM '94'' (Talence, 1994) and the
  ``Polyominoes and Tilings'' (Toulouse, 1994) Workshops.

\bibitem[Gor21]{VG}
Vadim Gorin.
\newblock {\em Lectures on random lozenge tilings}, volume 193 of {\em
  Cambridge Studies in Advanced Mathematics}.
\newblock Cambridge University Press, Cambridge, 2021.

\bibitem[Hat02]{H}
Allen Hatcher.
\newblock {\em Algebraic topology}.
\newblock Cambridge University Press, Cambridge, 2002.

\bibitem[JPS98]{JPS}
William Jockusch, James Propp, and Peter Shor.
\newblock Random domino tilings and the arctic circle theorem.
\newblock 02 1998.

\bibitem[Ken09]{K}
Richard Kenyon.
\newblock Lectures on dimers.
\newblock In {\em Statistical mechanics}, volume~16 of {\em IAS/Park City Math.
  Ser.}, pages 191--230. Amer. Math. Soc., Providence, RI, 2009.

\bibitem[KO05]{KO}
Richard Kenyon and Andrei Okounkov.
\newblock Limit shapes and the complex burgers equation.
\newblock {\em Acta Math.}, 199, 08 2005.

\bibitem[KSO03]{KOS}
Richard Kenyon, Scott Sheffield, and Andrei Okounkov.
\newblock Dimers and amoebae.
\newblock {\em Annals of mathematics, ISSN 0003-486X, Vol. 163, Nº 3, 2006,
  pags. 1019-1056}, 163, 11 2003.

\bibitem[Lan99]{Lang}
Serge Lang.
\newblock {\em Complex analysis}, volume 103 of {\em Graduate Texts in
  Mathematics}.
\newblock Springer-Verlag, New York, fourth edition, 1999.

\bibitem[LS77]{LS}
B.~F. Logan and L.~A. Shepp.
\newblock A variational problem for random {Y}oung tableaux.
\newblock {\em Advances in Math.}, 26(2):206--222, 1977.

\bibitem[Mei08]{JM}
John Meier.
\newblock {\em Groups, graphs and trees. {An} introduction to the geometry of
  infinite groups.}, volume~73 of {\em Lond. Math. Soc. Stud. Texts}.
\newblock Cambridge: Cambridge University Press, 2008.

\bibitem[PST16]{PST}
Igor Pak, Adam Sheffer, and Martin Tassy.
\newblock Fast domino tileability.
\newblock {\em Discrete and Computational Geometry}, 56, 09 2016.

\bibitem[PW96]{PW}
James~Gary Propp and David~Bruce Wilson.
\newblock Exact sampling with coupled {M}arkov chains and applications to
  statistical mechanics.
\newblock In {\em Proceedings of the {S}eventh {I}nternational {C}onference on
  {R}andom {S}tructures and {A}lgorithms ({A}tlanta, {GA}, 1995)}, volume~9,
  pages 223--252, 1996.

\bibitem[STCR95]{STCR}
N.~C. Saldanha, C.~Tomei, M.~A. Casarin, Jr., and D.~Romualdo.
\newblock Spaces of domino tilings.
\newblock {\em Discrete Comput. Geom.}, 14(2):207--233, 1995.

\bibitem[Thu90]{T}
William Thurston.
\newblock Conway's tiling groups.
\newblock {\em American Mathematical Monthly}, 97, 10 1990.

\bibitem[Val45]{V}
F.~A. Valentine.
\newblock A {L}ipschitz condition preserving extension for a vector function.
\newblock {\em Amer. J. Math.}, 67:83--93, 1945.

\bibitem[VK77]{VK}
A.~M. Vershik and S.~V. Kerov.
\newblock Asymptotics of the {Plancherel} measure of the symmetric group and
  the limiting form of {Young} tableaux.
\newblock {\em Sov. Math., Dokl.}, 18:527--531, 1977.

\end{thebibliography}

\subsection{Computer simulations of the modified Aztec diamond.}
Here, we present simulations of random domino tilings of $\mathcal{AD}_{N}$ made for a randomly-chosen height change $r$.

\label{simul}
\begin{figure}[ht!]
    \includegraphics[width=0.5\linewidth]{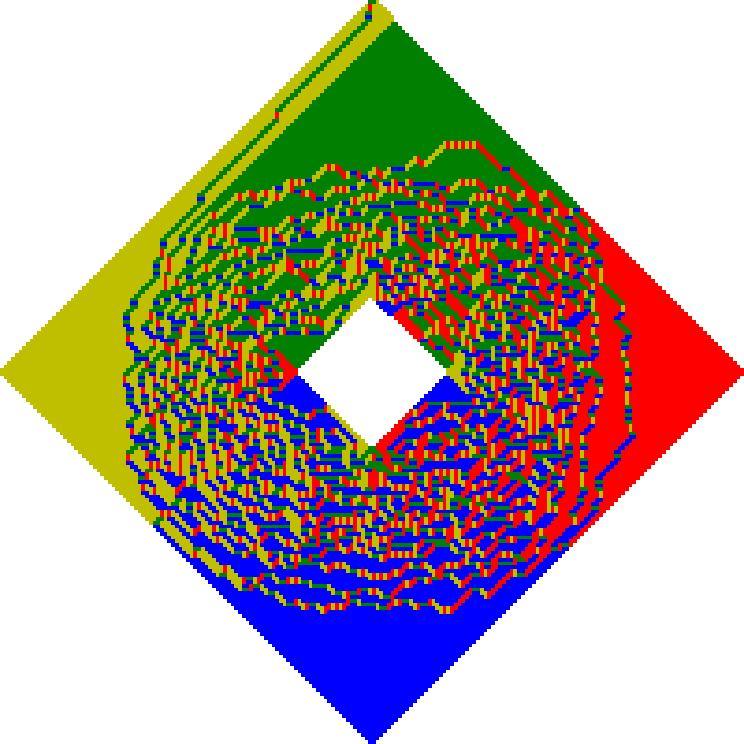}
    \caption{A domino tiling of $\mathcal{AD}_{50}$ with $M=50$ and $R=88$.}
\end{figure}
\begin{figure}[ht!]
    \includegraphics[width=0.5\linewidth]{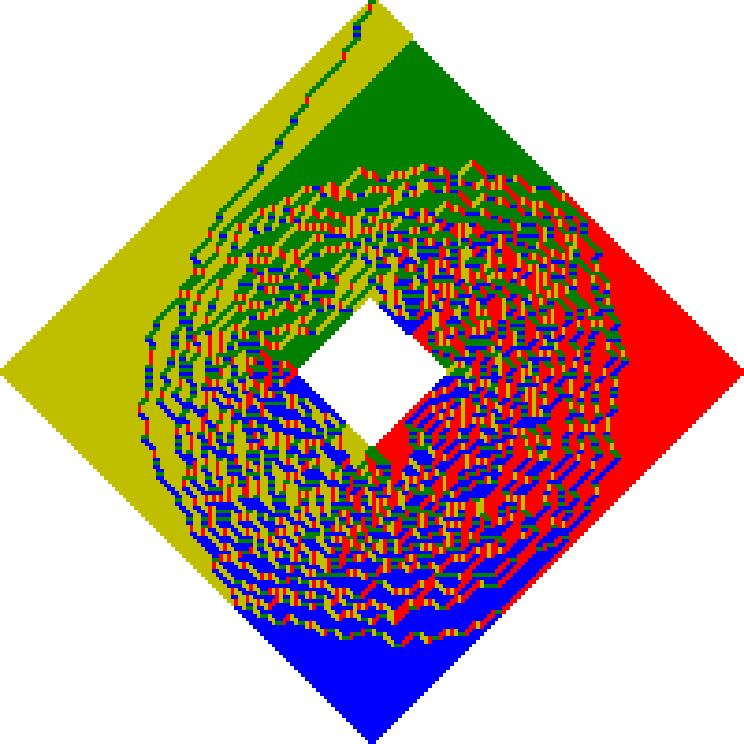}
    \caption{A domino tiling of $\mathcal{AD}_{50}$ with $M=100$ and $R=72$ }
\end{figure}
\pagebreak

\begin{figure}[ht!]
    \includegraphics[width=0.5\linewidth]{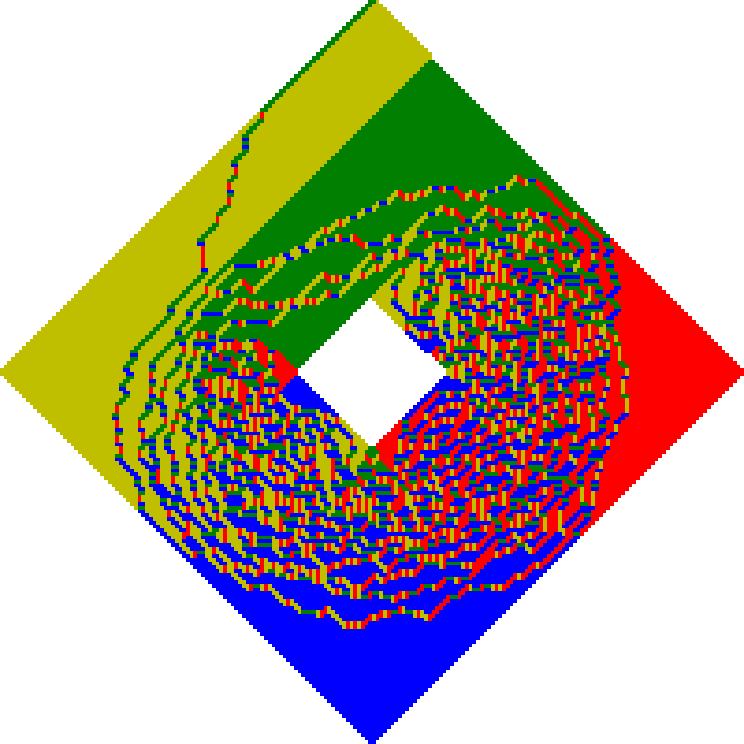}
    \caption{A domino tiling of $\mathcal{AD}_{50}$ with $M=150$ and $R=76$} 
\end{figure}
\begin{figure}[ht!]
    \includegraphics[width=0.5\linewidth]{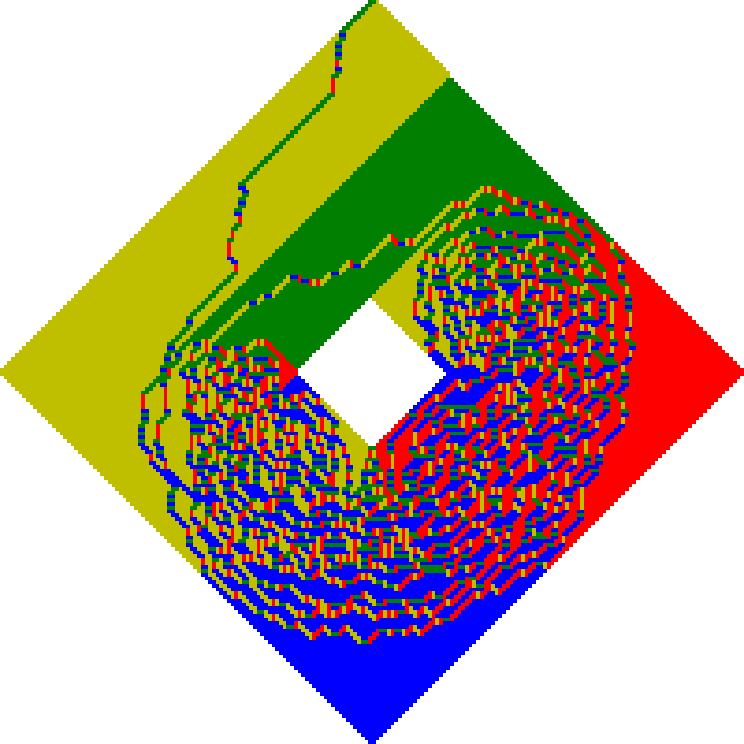}
    \caption{A domino tiling of $\mathcal{AD}_{50}$ with $M=200$ and $R=88$}
\end{figure}
\pagebreak

\begin{figure}[ht!]
    \centering
    \includegraphics[width=0.45\linewidth]{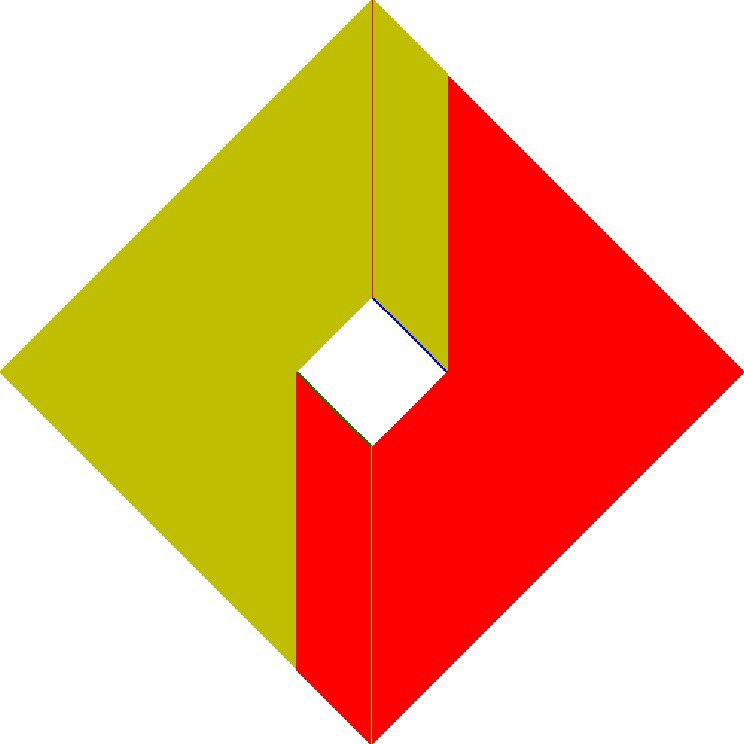}
    \caption{A domino tiling with the minimal height change $R=-300$. Almost all the dominoes are vertical.}
    \label{figure6}
\end{figure}
\begin{figure}[ht!]
    \centering
    \includegraphics[width=0.45\linewidth]{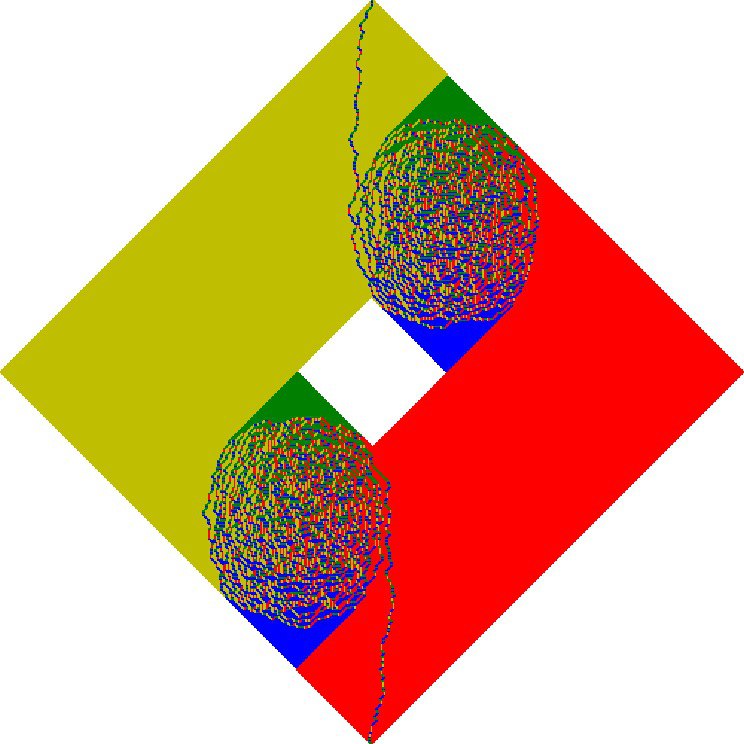}
    \caption{A typical domino tiling with the minimal height change $R=-300$.}
    \label{figure5}
\end{figure}

\end{document}